\documentclass{amsart}
\usepackage[usenames,dvipsnames]{color}
\usepackage{hyperref}
\usepackage{amssymb}
\usepackage{tikz}
\usetikzlibrary{calc}
\usepackage{tikz-cd}
\usepackage{tcolorbox}
\usepackage[alphabetic,backrefs]{amsrefs}
\usepackage{mathtools}

\usepackage{geometry}[letterpaper,margin=1in]
\numberwithin{equation}{section}
\newcommand{\ZZ}{\mathbb{Z}}
\newcommand{\NN}{\mathbb{N}}

\renewcommand{\tilde}{\widetilde}

\DeclareMathOperator{\lcm}{lcm}

\DeclareMathOperator{\Cl}{Cl}
\DeclareMathOperator{\link}{link}
\DeclareMathOperator{\cpxstar}{star}

\newcommand{\goodsimp}{\omega}
\newcommand{\xmax}{n}
\newcommand{\ymax}{k}
\newcommand{\bp}{\diamondsuit}
\newcommand{\newvert}{{v'}}
\newcommand{\redH}{\widetilde{H}}
\newcommand{\field}{\mathbf{k}}

\newcommand{\defn}[1]{\textit{#1}}

\newtheorem{theorem}[equation]{Theorem}
\newtheorem*{theorem*}{Theorem}
\newtheorem{lemma}[equation]{Lemma}

\newtheorem{cor}[equation]{Corollary}

\newtheorem{prop}[equation]{Proposition}
\newtheorem*{prop*}{Proposition}
\theoremstyle{remark}

\newtheorem{remark}[equation]{Remark}
\newtheorem{example}[equation]{Example}
\newtheorem{definition}[equation]{Definition}

\begin{document}
\title{Subdividing simplicial virtual resolutions with homology}

\author{Eric Nathan Stucky}
\author{Jay Yang}
\begin{abstract}
  While sporadic examples of virtual resolutions with homology have been constructed,
  their occurrence is not well understood or controlled.
  Our results build a new set of tools for studying virtual resolutions of monomial ideals as arising from simplicial complexes, including characterizing them by the acyclicity of certain induced subcomplexes.
  Using this characterization, we give a description of minimal simplicial complexes supporting virtual resolutions as well as a technique for removing homology from simplicial virtual resolutions.
\end{abstract}
\maketitle

\section{Introduction}

Virtual resolutions, as defined in \cite{virtual-res}*{Definition~1.1} form an important tool in the study of syzygies over toric varieties other than $\mathbb{P}^n$.
For a toric variety $X$, we can associate a polynomial ring $S$ called the Cox ring, and a monomial ideal $B$ called the irrelevant ideal. 
Virtual resolutions are complexes of free modules $F_\bullet$ such that $H_i(F_\bullet)$ for $i>0$ must be annihilated by a sufficiently high power of the irrelevant ideal $B$.
From a more geometric perspective, for a module $M$, we say that $F_{\bullet}$ is a virtual resolution of $M$ if the corresponding complex of sheaves $F_{\bullet}^{\sim}$ is a free resolution of $M^\sim$. We defer the more detailed discussion of this definition to Section~\ref{sec:virtual-resolutions}.

Among current constructions of virtual resolutions, almost all seem to produce free resolutions in practice. Many constructions proceed explicitly by giving a free resolution of a different module, exploiting the fact that the correspondence between sheaves and modules is many-to-one. This includes the constructions from \cite{virtual-res} and \cite{HHL24} as well as the monomial constructions from \cites{virtual-cm-paper,virtual-shelling}. Some of the few documented exceptions are \cite{lauren-truncation}*{Example~6.7} as well as the examples in \cite{virtual-eagon-northcott}*{Section~4}.

The lack of examples of virtual resolutions with homology and the lack of results addressing the occurrence of homology in virtual resolutions represents a critical gap in our understanding of virtual resolutions. This paper contributes an approach to this question in the monomial case by extending the work on simplicial and cellular resolutions from \cites{bps-monomial-resolutions,bs-cellular} to study virtual resolutions.

For a simplicial complex $\Delta$ with simplices labeled by monomials, one can associate a chain complex $F_\Delta$ using the techniques of the papers \cite{bps-monomial-resolutions}. We describe this construction and its interaction with virtual resolutions in more detail in Section~\ref{sec:virtual-resolutions}. Just as \cite{bps-monomial-resolutions}*{Lemma~2.2} gives an acyclicity condition describing when a labeled simplicial complex gives rise to a free resolution, we can state the following description of when the chain complex $F_\Delta$ is a virtual resolution.

\begin{theorem}
  \label{thm:virtual-resolution}
  Fix a smooth toric variety $X$ with Cox ring $S$ and irrelevant ideal $B$.
  For a labeled simplicial complex $(\Delta,\ell)$, 
  its associated chain complex $F_\Delta$ is a virtual resolution if and only if
  for each of the subcomplexes \[\Delta_m:=\{\sigma \in \Delta : \ell(\sigma)\text{ divides }m\},\]
  there exists some $d\geq 0$ such that $\redH_i(\Delta_m;\field)=0$ for all monomials $m\in B^{d}$ and integers $i\geq 0$.
\end{theorem}

This description then allows us to construct examples of virtual resolutions with homology. For now, the following example serves as a good prototype of such a virtual resolution. Sections~\ref{sec:minimal-complexes} and \ref{sec:products-case} expand more on the ideas behind this example to produce a large class of examples of virtual resolutions with homology.

\begin{example}
\label{ex:intro-virtual}
Recall that when $X$ is the product of projective spaces $\mathbb{P}^2\times\mathbb{P}^1$ it has Cox ring $S=\field[x_0,x_1,x_2,y_0,y_1]$ and irrelevant ideal $B=\langle x_0,x_1,x_2\rangle \cap \langle y_0,y_1\rangle$. Consider the following labeled simplicial complex $(\Delta,\ell)$ on vertices $\{v_0,v_1,w_0,w_1\}$ depicted in Figure~\ref{fig:intro-virtual}.

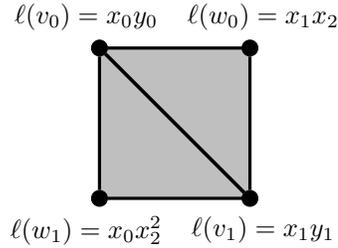
\begin{figure}[h!]
  \[
  \begin{tikzpicture}
    \fill[lightgray] (0,0) -- (2,0) -- (0,2) -- cycle;
    \fill[lightgray] (2,2) -- (2,0) -- (0,2) -- cycle;
    \node[circle, draw, inner sep=0pt, minimum size=6pt, fill=black, label=below:{$\ell(w_1)=x_0x_2^2\quad$}] (2) at (0,0) {};
    \node[circle, draw, inner sep=0pt, minimum size=6pt, fill=black, label=above:{$\quad\ell(w_0)=x_1x_2$}] (3) at (2,2) {};
    \node[circle, draw, inner sep=0pt, minimum size=6pt, fill=black, label=above:{$\ell(v_0)=x_0y_0\quad$}] (4) at (0,2) {};
    \node[circle, draw, inner sep=0pt, minimum size=6pt, fill=black, label=below:{$\quad\ell(v_1)=x_1y_1$}] (5) at (2,0) {};
    \draw[very thick] (5) -- (4) -- (2) -- (5) -- (3) -- (4) -- (5);
  \end{tikzpicture}
  \]
  \caption{A labeled simplicial complex giving a virtual resolution with homology}
  \label{fig:intro-virtual}
\end{figure}

For any monomial $m\in B$, the associated subcomplex $\Delta_m$ is contractible or empty. In particular, the only induced subcomplex of $\Delta$ with non-trivial homology is the subcomplex containing only the vertices $w_0$ and $w_1$. If $\ell(w_0)\mid m$ and $\ell(w_1)\mid m$, then $m$ is divisible by each of $x_0$, $x_1$, and $x_2$. If $m\in B$, then $m$ is divisible by either $y_0$ or $y_1$; thus for $m\in B$, the complex $\Delta_m$ is either empty or contains at least one of the vertices $v_0$ or $v_1$. Thus by Theorem~\ref{thm:virtual-resolution} the resulting complex $F_\Delta$ is a virtual resolution.

On the other hand, $F_\Delta$ is not a free resolution as $\Delta_{x_0x_1x_2^2}$ is exactly the induced subcomplex on $w_0$ and $w_1$, and is thus not acyclic. Associated to the labeled simplicial complex $(\Delta,\ell)$, we get the following chain complex, denoted $F_\Delta$:
\[
S
  \xleftarrow{%
    \begin{bmatrix}
      x_0y_0 & x_1x_2 & x_0x_2^2 & x_1y_1
    \end{bmatrix}}
S^4
  \xleftarrow{%
    \begin{bmatrix}
      -x_1x_2 & -x_2^2 & -x_1y_1 &  &  \\
      x_0y_0 &  &  & -y_1 &  \\
       & y_0 &  &  & -x_1y_1 \\
       &  & x_0y_0 & x_2 & x_0x_2^2
    \end{bmatrix}}
S^5 
  \xleftarrow{%
    \begin{bmatrix}
      y_1 &  \\
       & x_1y_1 \\
      -x_2 & -x_2^2 \\
      x_0y_0 &  \\
       & y_0
    \end{bmatrix}}%
S^2.
\]

This complex then is a virtual resolution of $S/I$ where $I=\langle  x_0y_0,x_1y_1,x_1x_2,x_0x_2^2\rangle$ is the ideal generated by the vertex labels. Checking homology, we find that $H_1(F_\Delta)$ is non-zero and is in fact generated in the degree corresponding to the monomial $x_0x_1x_2^2$. We check this assertion in more detail in Example~\ref{ex:intro-redo}.
\end{example}

A second key contribution of this paper is the development of a new technique to remove homology from virtual resolutions arising from labeled simplicial complexes.
Relatively uniquely among techniques for virtual resolutions, this technique operates on the resolution via the simplicial complex instead of modifying the module.
This technique begins with a stellar subdivision of a simplicial complex at one of its faces along with an appropriate choice of label on the new vertex.

\begin{definition}
  Fix a labeled simplicial complex $(\Delta,\ell)$ with labels in a polynomial ring $S$, a simplex $\sigma\in \Delta$, and a monomial $m\in S$.
  Define a \defn{subdivision of $(\Delta,\ell)$ at $\sigma$}
  as a pair $(\Delta',\ell')$ where $\Delta'$ is the subdivision of
  $\Delta$ at the simplex $\sigma$ by a single vertex
  and where $\ell'(\tau)=\ell(\tau)$ for all $\tau\in \Delta\cap \Delta'$.
\end{definition}

One way to label the new vertex from this subdivision would be to give it the label $\ell(\sigma)$ of the cell being subdivided.
This labeling doesn't change the resolution in a useful way, as all it serves to do is to make the resolution more non-minimal.
However, with virtual resolutions, we can make a more useful change by making an appropriate choice of monomial label on the new vertex arising from the subdivision.

\begin{theorem}
  \label{thm:subdivision-virtual}
  Fix a smooth toric variety $X$ with Cox ring $S$ and irrelevant ideal $B$.
  Let $(\Delta,\ell)$ be an lcm-labeled simplicial complex such that $F_\Delta$ is a
  virtual resolution of $S/I$ for a monomial ideal $I\subseteq S$.
  Suppose that $\goodsimp\in \Delta$ is a simplex with vertices $v_0,\ldots,v_r$ and
  let $\Delta'$ be the stellar subdivision of $\Delta$ at a point $\newvert$ in $\goodsimp$.
  If $\ell'$ is a lcm-labeling on $\Delta'$ such that
\begin{enumerate}
\item\label{esmall} $\ell'(\newvert)$ divides $\ell(\goodsimp)$, and
\item\label{elarge} $\ell'(\newvert)\in \langle \ell(v_0),\ldots,\ell(v_r)\rangle : B^{\infty}$,
\end{enumerate}
then $F_{\Delta'}$ is a virtual resolution of $S/I$.
\end{theorem}

We will call a subdivision satisfying the conditions of Theorem~\ref{thm:subdivision-virtual} a \defn{virtual compatible subdivision}. Aside from simply giving a new and unique way to construct new resolutions from old resolutions, this technique often reduces the homology in the virtual resolution. The earlier example admits such a subdivision that is again illustrative.

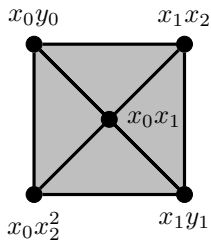
\begin{figure}[h!]
  \[
  \begin{tikzpicture}
    \fill[lightgray] (0,0) -- (2,0) -- (0,2) -- cycle;
    \fill[lightgray] (2,2) -- (2,0) -- (0,2) -- cycle;
    \node[circle, draw, inner sep=0pt, minimum size=6pt, fill=black, label=below:$x_0x_2^2$] (2) at (0,0) {};
    \node[circle, draw, inner sep=0pt, minimum size=6pt, fill=black, label=above:$x_1x_2$] (3) at (2,2) {};
    \node[circle, draw, inner sep=0pt, minimum size=6pt, fill=black, label=above:$x_0y_0$] (4) at (0,2) {};
    \node[circle, draw, inner sep=0pt, minimum size=6pt, fill=black, label=below:$x_1y_1$] (5) at (2,0) {};
    \node[circle, draw, inner sep=0pt, minimum size=6pt, fill=black, label=right:$x_0x_1$] (6) at (1,1) {};
    \draw[very thick] (2) -- (3);
    \draw[very thick] (5) -- (4) -- (2) -- (5) -- (3) -- (4) -- (5); 
  \end{tikzpicture}
  \]
  \caption{The labeled simplicial complex obtained by a virtual compatible subdivision of the labeled simplicial complex from Figure~\ref{fig:intro-virtual}}
  \label{fig:intro-subdivision}
\end{figure}

\begin{example}
\label{ex:intro-subdivision}
Consider again the labeled complex $(\Delta, \ell)$ from Example~\ref{ex:intro-virtual},  and subdivide it at a point on the edge between $v_0$ and $v_1$ to produce the subdivision $(\Delta',\ell')$ depicted in Figure~\ref{fig:intro-subdivision}:

Observe that this is a virtual compatible subdivision. The label of the new vertex $\ell'(\newvert)=x_0x_1$ divides $\lcm(\ell(v_0),\ell(v_1))=x_0x_1y_0y_1$, the label on the edge between $v_0$ and $v_1$. Also, it is contained in the ideal $\langle \ell(v_0),\ell(v_1) \rangle : B^\infty$; in fact it is the generator of this ideal $\langle x_0x_1 \rangle$.

To each of these labeled simplicial complexes we again associate a chain complex $F_{\Delta'}$ which this time has the form
\[
S
  \xleftarrow{{ %
    \begin{bsmallmatrix}
      x_0y_0 \\ x_1x_2 \\ x_0x_2^2 \\ x_1y_1 \\ x_0x_1
    \end{bsmallmatrix}^{\!\!\top}}}
S^5
  \xleftarrow{{
    \begin{bsmallmatrix}
      -x_1x_2 & -x_2^2 & -x_1 &  &  &  &  &  \\
      x_0y_0 &  &  & -y_1 & -x_0 &  &  &  \\
       & y_0 &  &  &  & -x_1y_1 & -x_1 &  \\
       &  &  & x_2 &  & x_0x_2^2 &  & -x_0 \\
       &  & y_0 &  & x_2 &  & x_2^2 & y_1
    \end{bsmallmatrix}}}
S^8
  \xleftarrow{ {%
    \begin{bsmallmatrix}
      1 &  &  &  \\
       & x_1 &  &  \\
      -x_2 & -x_2^2 &  &  \\
       &  & x_0 &  \\
      y_0 &  & -y_1 &  \\
       &  &  & 1 \\
       & y_0 &  & -y_1 \\
       &  &  x_2 & x_2^2
    \end{bsmallmatrix}}}%
S^4.
\]  
Checking homology, one finds that $F_{\Delta'}$ is in fact a free resolution. In particular it is a free resolution of $S/J$ where $J=\langle x_0x_1,x_0y_0,x_1x_2,x_1y_1,x_0x_2^2 \rangle$. A quick computation verifies that the sheaves $(S/I)^\sim$ and $(S/J)^\sim$ are isomorphic, and so per \cite{virtual-res}*{Definition~1.1}, $F_{\Delta'}$ is a virtual resolution of $S/I$. 
\end{example}

Our main result describes when such a subdivision procedure can reduce the homology in a specified degree. 

\begin{theorem}
  \label{thm:main-theorem}
  Fix a smooth toric variety $X$ with Cox ring $S$.
  Let $(\Delta,\ell)$ be an lcm-labeled simplicial complex such that $F_\Delta$ is a virtual resolution, 
  and $(\Delta',\ell')$ be an lcm-labeled virtual compatible subdivision of $(\Delta,\ell)$ at a point $\newvert$ in a simplex $\goodsimp\in \Delta$.
  Write $\Gamma:=\link_{\omega}(\Delta)$.

  If for every monomial $m\in S$ where $\ell'(\newvert)$ divides $m$ and $\ell(v)$ does not divide $m$ for any $v\in \omega$, the induced map on homology $\redH_{j}(\Gamma_m;\field)\rightarrow \redH_{j}(\Delta_m;\field)$ is an injection for all $j\geq -1$,
  then for all $i\geq  0$ and $\alpha\in \ZZ^n$
  \[\dim H_{i}(F_{\Delta'})_\alpha\leq \dim H_{i}(F_\Delta)_{\alpha}.\]
  Moreover, the inequality is strict for all  $i\geq 0$ and $\alpha\in \ZZ^n$ where $\ell'(\newvert)$ divides $m:=x^\alpha$ and $\redH_{i-1}(\Gamma_m;\field)\neq 0$.
\end{theorem}

In summary, subject to some divisibility and topological criteria, the subdivision of $\Delta$ at a point in a simplex $\goodsimp$ yields a virtual resolution with less homology.
As an application, using these tools we can prove the following statement about a specific class of simplicial virtual resolutions.

\begin{theorem}
  \label{thm:star-trivial}
  Let $X=\mathbb{P}^n\times\mathbb{P}^k$ with $k\leq n$ and let $S$ be the Cox ring of $X$.
  Let $(\Delta,\ell)$ be an lcm-labeled simplicial complex with $\Delta$ the $(k+1)$-dimensional bipyramid with exactly two maximal simplices.
  If $F_{\Delta}$ is a virtual resolution of $S/I$ for a monomial ideal $I\subseteq S$, then there exists a labeled simplicial complex $(\Delta',\ell')$ such that
  $F_{\Delta'}$ is a virtual resolution of $S/I$ where $H_i(F_{\Delta'})=0$ for all $i>0$.
\end{theorem}

\section{Virtual Resolutions from Simplicial Complexes}
\label{sec:virtual-resolutions}

We will work exclusively with simplicial complexes throughout this paper, but the key results of this section can be generalized to cell complexes without significant difficulty. With more care many of the other results in this paper can be generalized. Throughout this paper, $\field $ denotes a field and all polynomial rings are over $\field $. As we are working with monomial ideals, we will use the fine grading on all polynomial rings unless otherwise mentioned. In particular, the polynomial ring $S=\field[x_1,\ldots,x_n]$ has the $\ZZ^n$-grading given by setting the degree of each variable to the corresponding standard basis element. Moreover, for an element $\alpha$ of $\NN^n$, we denote the corresponding monomial by $x^{\alpha}$.

We begin by standardizing some notation about simplicial complexes. Recall that a simplicial complex $\Delta$ on a finite set of vertices $V$ is a subset of $2^V$ such that if $\sigma\in \Delta$ and $\tau\subseteq \sigma$ then $\tau\in \Delta$. In particular, every non-empty simplicial complex contains the empty simplex and we will consider the empty simplex to be of dimension $-1$. As an abuse of notation, we will sometimes consider a single simplex $\sigma\subseteq V$ as the simplicial complex $2^\sigma$. 

We define several standard operations for simplicial complexes.
For two simplicial complexes $\Delta$ and $\Gamma$ on disjoint vertex sets, we denote the \defn{join} of $\Delta$ and $\Gamma$ by
\[\Delta\star \Gamma := \{\sigma \sqcup \tau : \sigma\in \Delta, \tau\in \Gamma\}.\]

For a simplex $\tau$ in $\Delta$, denote the closed \defn{star} of $\tau$ in $\Delta$ by
\[\cpxstar_{\tau}(\Delta) := \{\sigma \in \Delta : \tau\cup \sigma\in \Delta\}.\]
Similarly, denote the \defn{link} of a simplex $\tau$ in $\Delta$ by
\[\link_{\tau}(\Delta) := \{\sigma \in \Delta : \tau\cap \sigma = \emptyset, \tau\cup \sigma \in \Delta\}.\]
We also observe here that the link is the maximal subcomplex of $\Delta$ satisfying $\link_{\sigma}(\Delta) \star \sigma \subseteq \Delta$.

Throughout this paper, a subdivision of a simplicial complex will refer to a stellar subdivision at a single point in a simplex.
That is, for a simplex $\goodsimp\in \Delta$, we define the \defn{subdivision of $\Delta$ at a point $\newvert$ in $\goodsimp$} as the simplicial complex $\Delta'$ satisfying
\[\Delta' = \left(\Delta \setminus \{\sigma \in \Delta : \goodsimp\subseteq \sigma\}\right) \cup \{\sigma \cup \{\newvert\} : \sigma\cup \goodsimp \in \Delta, \goodsimp\not\subseteq\sigma \}.\]

We also define here a simplicial complex we call the \emph{bipyramid over a $k$-simplex}, which we denote $\bp^k$, as the join of a pair of vertices and a $k$-simplex. While there are other triangulations of a bipyramid, this is the one of relevance for this paper. 

\subsection*{Labeled Simplicial Complexes}

We define the following notion of a labeled simplicial complex, which is a slightly modified version of the construction that \cite{bps-monomial-resolutions} describes.
\begin{definition}[\cite{bps-monomial-resolutions}*{Construction~2.1}]
  \label{def:labeled-simplicial-complex}
Fix a polynomial ring $S$. A \defn{labeled simplicial complex} is a pair $(\Delta,\ell)$, where $\Delta$ is a simplicial complex, and $\ell:\Delta\to S$ is a \defn{labeling} satisfying
\begin{enumerate}
\item $\ell(\sigma)$ is a monomial for each simplex $\sigma\in\Delta$, and
\item $\ell(\sigma)\mid \ell(\gamma)$ for all inclusions $\sigma\subset \gamma$.
\end{enumerate}
Moreover, we will say a labeling is an \emph{lcm-labeling} if $\ell(\emptyset)=1$ and for each non-empty simplex $\sigma$,
  \[\ell(\sigma)=\lcm\{\ell(\tau) : \tau\subsetneq \sigma\}.\]
\end{definition}

Where it is unambiguous, we will refer to the labeled simplicial complex by $\Delta$, and omit mention of $\ell$.

\begin{definition}
  \label{def:labeled-simplicial-chain-complex}
  Given a labeled simplicial complex $(\Delta,\ell)$, we can associate a free complex $F_\Delta$ where for $i\geq 0$,
  \[ F_{\Delta,i} = \bigoplus_{\sigma\in \Delta,~ \dim \sigma = i-1} S(-\deg(\ell(\sigma))).\]
  The differentials for this complex $\partial_i: F_{\Delta,i}\to F_{\Delta, i-1}$ are given by
  \[ \partial_i [\sigma] = \sum_{\tau\subseteq\sigma,~ \dim \tau = \dim \sigma - 1} \pm\frac{\ell(\sigma)}{\ell(\tau)}[\tau],\]
  where the sign is chosen based on the relative orientation of $\tau$ and $\sigma$.
\end{definition}

Condition (2) of Definition~\ref{def:labeled-simplicial-complex} is sufficient to imply that the collection of modules and maps in Definition~\ref{def:labeled-simplicial-chain-complex} are well-defined. A computation similar to that for simplicial homology shows that $F_\Delta$ always forms a complex of $S$-modules.

So as long as $\ell(\emptyset)=1$, the first term in the complex $F_\Delta$ is given by $F_{\Delta,0}=S$.
Consequently, for such complexes, $H_0(F_\Delta)=S/I$ where $I$ is the ideal generated by the labels on all vertices of $\Delta$.
More generally, the homology of the complex $F_\Delta$ can be determined using the homology of the following subcomplexes of $\Delta$ associated to a monomial $m$:
\[\Delta_m := \{\sigma\in \Delta : \ell(\sigma) \text{ divides } m\}.\]
The key connection behind the theory of simplicial resolutions is as follows.

\begin{theorem}[\cite{bps-monomial-resolutions}]
  \label{thm:cellular-subcomplexes}
  Fix a polynomial ring $S$. Let $(\Delta,\ell)$ be a labeled simplicial complex with labels in $S$, and let $F_{\Delta}$ be the free complex associated to this labeling. For any integer $i\geq 0$ and any monomial $m=x^{\alpha}\in S$,
  \[H_{i}(F_{\Delta})_{\alpha} \cong \widetilde{H}_{i-1}(\Delta_m;\field).\]
\end{theorem}
\begin{proof}
  This follows the technique of the proof of \cite{bps-monomial-resolutions}*{Lemma~2.2}. Namley, we can observe that the degree $\alpha$ strand of the complex $F_{\Delta}$ is exactly the cellular chain complex with coefficients in $\field $ for $\Delta_m$. As such we get the equality $H_{i}(F_{\Delta})_{\alpha}=H_i(\Delta_m;\field)$.
\end{proof}

Recall that a free complex $F_\bullet$ supported in non-negative homological degrees is a \defn{free resolution} if $H_i(F_\bullet)=0$ for all $i>0$; in this case we say it is a free resolution \defn{of} $H_0(F_\bullet)$. Thus, in the standard theory of cellular resolutions, this allows us to characterize the acyclicity of the resolution as the acyclicity of the subcomplexes $\Delta_m$. 

\subsection*{Virtual Resolutions}

We recall some background on toric varieties; for reference, we refer the reader to the book \cite{toric-varieties}. For a normal toric variety $X$, we can associate a $\Cl(X)$-graded polynomial ring $S$ called the Cox ring, along with a square-free monomial ideal $B\subseteq S$. To any finitely generated $\Cl(X)$-graded $S$-module $M$ there is a unique associated sheaf $M^\sim$ on $X$;
this correspondence is in general many-to-one.

Critically, for a smooth toric variety $X$, the finitely generated modules that give the zero sheaf are exactly the $B$-torsion modules. That is to say, $M^\sim = 0$ if and only if there exists $d\geq 0$ such that $B^d\cdot M = 0$. As the irrelevant ideal $B$ is always a monomial ideal, this implies that if $S=\field[x_1,\ldots,x_n]$ and $M$ is $\ZZ^n$-graded then it suffices to check that $M_\alpha=0$ for all monomials $x^{\alpha}\in B^d$ for some $d\geq 0$.

Now with this context, we can recall the definition of a virtual resolution from \cite{virtual-res}.
\begin{definition}[\cite{virtual-res}*{Definition~1.1}]
    Fix a smooth toric variety $X$ with Cox ring $S$.
  Let $M$ be a $\Cl(X)$-graded $S$-module. A \defn{virtual resolution} of $M$ is a $\Cl(X)$-graded complex of free $S$-modules $F_\bullet$ such that the corresponding complex of vector bundles $F_\bullet^{\sim}$ is a locally-free resolution of $M^\sim$.
\end{definition}

\begin{remark}
  \label{rem:vres-saturation}
  To check if a virtual resolution $F_\bullet$ is a virtual resolution of a module $M$, we must in general show that $H_0(F_\bullet)^\sim = M^\sim$. However, for this paper, there is one important case, namely where $M$ and $H_0(F_\bullet)$ are both presented as quotients of $S$ by a homogeneous ideal. Then over a smooth toric variety $X$, by \cite{toric-varieties}*{Proposition~6.A.7},  $(S/I)^\sim=(S/J)^\sim$ if and only if $I:B^\infty=J:B^\infty$ where $B$ is the irrelevant ideal for $X$. As all of our complexes in this paper arise as $F_\Delta$ for an lcm-labeled simplicial complex $\Delta$, the module $H_0(F_\Delta)$ can always be presented as a quotient of $S$ by a homogeneous ideal.
\end{remark}

As with free resolutions, whether the complex $F_\Delta$ associated to the labeled simplicial complex $(\Delta,\ell)$ is a virtual resolution can be characterized by an acyclicity condition on subcomplexes. We recall here the statement of Theorem~\ref{thm:virtual-resolution} for symmetry with Theorem~\ref{thm:cellular-subcomplexes}.

\begin{theorem*}[Theorem~\ref{thm:virtual-resolution}]
  Fix a smooth toric variety $X$ with Cox ring $S=\field[x_1,\ldots,x_n]$ and irrelevant ideal $B$. Let $(\Delta,\ell)$ be a labeled simplicial complex.
  The complex $F_\Delta$ is a virtual resolution if and only if
  there exists some $d\geq 0$ such that $\redH_j(\Delta_m;\field)=0$ for all $m\in B^{d}$ for $j\geq 0$.
\end{theorem*}
\begin{proof}[Proof of Theorem~\ref{thm:virtual-resolution}]
  If $F_{\Delta}$ is a virtual resolution, then as $F_{\Delta}$ is $\ZZ^n$-graded,
  there exists $d\geq 0$ such that for each $i>0$, we have $H_i(F_\Delta)_\alpha=0$
  for all $\alpha\in \ZZ^n$ with $x^{\alpha}\in B^d$.
  By Theorem~\ref{thm:cellular-subcomplexes}, for $j=i-1$, we find that $\redH_j(\Delta_m;\field)=0$ for all $m\in B^{d}$.
  
  Now for the reverse direction, fix $d\geq 0$ such that that $\widetilde{H}_j(\Delta_m;\field)=0$ for all $m\in B^{d}$ for $j\geq 0$. We must show that $H_i(F_{\Delta}^{\sim})=0$. Since $H_i(F_{\Delta}^{\sim}) \cong (H_i(F_\Delta))^{\sim}$, so it suffices to show that $B^{d}\cdot H_i(F_\Delta) = 0$.

  Let $m=x^{\alpha}\in B^{d}$ be a monomial, and let $\mu_m : H_i(F_\Delta) \rightarrow H_i(F_\Delta)$ be the map given by multiplication by $m$. Then it suffices to show that $\mu_m$ is zero. As $H_i(F_\Delta)$ is generated in degrees corresponding to the labels on the $i$-dimensional cells of $\Delta$, we find that it is generated in non-negative $\ZZ^n$-degrees. Thus the image of $\mu_m$ is contained in the components of degree at least $\alpha$. Theorem~\ref{thm:cellular-subcomplexes} gives $H_i(F_\Delta)_{\alpha'}=\widetilde{H}_{i-1}(\Delta_{m'};\field)$ for $m'=x^{\alpha'}$. When $m'\in B^{d}$, then equivalently $\alpha'\geq\alpha$, so these vector spaces are $0$. Thus the image of $\mu_m$ is zero, and so $(H_i(F_\Delta))^{\sim}=0$.
\end{proof}

\begin{example}
\label{ex:intro-redo}
We now expand Example~\ref{ex:intro-virtual} with further detail. As before, we fix $X=\mathbb{P}^2\times\mathbb{P}^1$, which has Cox ring $S=\field[x_0,x_1,x_2,y_0,y_1]$ and irrelevant ideal $B=\langle x_0,x_1,x_2\rangle \cap \langle y_0,y_1\rangle$. Consider the following labeled simplicial complex $(\Delta,\ell)$ on vertices $\{v_0,v_1,w_0,w_1\}$:

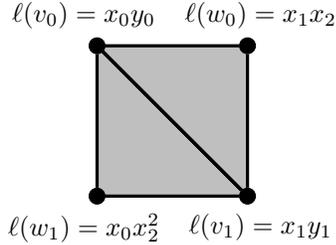
  \[
  \begin{tikzpicture}
    \fill[lightgray] (0,0) -- (2,0) -- (0,2) -- cycle;
    \fill[lightgray] (2,2) -- (2,0) -- (0,2) -- cycle;

    \node[circle, draw, inner sep=0pt, minimum size=6pt, fill=black, label=below:{$\ell(w_1)=x_0x_2^2\quad$}] (2) at (0,0) {};
    \node[circle, draw, inner sep=0pt, minimum size=6pt, fill=black, label=above:{$\quad\ell(w_0)=x_1x_2$}] (3) at (2,2) {};
    \node[circle, draw, inner sep=0pt, minimum size=6pt, fill=black, label=above:{$\ell(v_0)=x_0y_0\quad$}] (4) at (0,2) {};
    \node[circle, draw, inner sep=0pt, minimum size=6pt, fill=black, label=below:{$\quad\ell(v_1)=x_1y_1$}] (5) at (2,0) {};
    \draw[very thick] (5) -- (4) -- (2) -- (5) -- (3) -- (4) -- (5);
  \end{tikzpicture}
\]

Associated to the simplicial complex above, ignoring the labels, one gets a simplicial chain complex as follows:
\[
\field
  \xleftarrow{%
    \begin{bmatrix}
      1 & 1 & 1 & 1
    \end{bmatrix}}
\field^4
  \xleftarrow{%
    \begin{bmatrix}
      -1 & -1 & -1 &  &  \\
      1 &  &  & -1 &  \\
       & 1 &  &  & -1 \\
      &  & 1 & 1 & 1
    \end{bmatrix}}
\field^5 
  \xleftarrow{%
    \begin{bmatrix}
       1 &  \\
        & 1 \\
       -1 & -1 \\
       1 &  \\
        & 1
    \end{bmatrix}}%
\field^2
\]

Then the construction in Definition~\ref{def:labeled-simplicial-chain-complex} corresponds to replacing $\field$ with appropriate degree-shifted copies of $S$ and replacing the non-zero entries of the differentials by the appropriate monomial. This procedure then yields the following complex of $S$-modules; here again we elide the degree shifts.
\[
S
  \xleftarrow{%
    \begin{bmatrix}
      x_0y_0 & x_1x_2 & x_0x_2^2 & x_1y_1
    \end{bmatrix}}
S^4
  \xleftarrow{%
    \begin{bmatrix}
      -x_1x_2 & -x_2^2 & -x_1y_1 &  &  \\
      x_0y_0 &  &  & -y_1 &  \\
       & y_0 &  &  & -x_1y_1 \\
       &  & x_0y_0 & x_2 & x_0x_2^2
    \end{bmatrix}}
S^5 
  \xleftarrow{%
    \begin{bmatrix}
      y_1 &  \\
       & x_1y_1 \\
      -x_2 & -x_2^2 \\
      x_0y_0 &  \\
       & y_0
    \end{bmatrix}}%
S^2
\]

Aside from $H_0(F_\Delta)$, which is $S/I$ for $I=\langle x_0y_0,x_1y_1,x_1x_2,x_0x_2^2\rangle$, we additionally have $H_1(F_\Delta)$ which is generated by $\eta = -x_0x_2[w_0]+x_1[w_1]$, where $[w_0]$ and $[w_1]$ are the homology classes in $F_\Delta$ corresponding to the vertices $w_0$ and $w_1$. Moreover, $\eta$ is annihilated by $y_0$ and $y_1$ in $H_1(F_\Delta)$ so $H_1(F_\Delta)$ is annihilated by the irrelevant ideal $B=\langle x_0,x_1,x_2\rangle\cap \langle y_0,y_1 \rangle$.

\end{example}

We finish the section with the following two corollaries which will reduce the number of monomials that must be checked to show that $F_\Delta$ is a virtual resolution. These are of occasional convenience throughout the paper. In particular, Corollary~\ref{cor:annihilator-of-homology} is often useful when the homology is generated by a single element.
\begin{cor}
  \label{cor:bracket-powers}
  Fix a smooth toric variety $X$ with Cox ring $S$ and irrelevant ideal $B$. Let $(\Delta,\ell)$ be a labeled simplicial complex.
  The complex $F_\Delta$ is a virtual resolution if and only if there exists some $d\geq 0$ such that
  $\widetilde{H}_i(\Delta_m; \field)=0$ for all $i\geq 0$ and $m\in B^{[d]}=\langle m^d : m\in B\rangle$. 
\end{cor}
\begin{proof}
  This follows from Theorem~\ref{thm:virtual-resolution} along with the two facts: $B^{[d]}\subseteq B^{d}$, and $B^{D}\subseteq B^{[d]}$ for $D\gg d$.
\end{proof}


\begin{cor}
  \label{cor:annihilator-of-homology}
  Fix a smooth toric variety $X$ with Cox ring $S$ and irrelevant ideal $B$. Let $(\Delta,\ell)$ be a labeled simplicial complex. The complex $F_{\Delta}$ is a virtual resolution if and only if for every monomial $m\in S$ and $i\geq 0$ with $\redH_i(\Delta_m;\field)\neq 0$, there exists $d\geq 0$ such that $\widetilde{H}_i(\Delta_{b^dm};\field)=0$ for every minimial generator $b\in B$.
\end{cor}
\begin{proof}
  Follows from Corollary~\ref{cor:bracket-powers}.
\end{proof}

\section{Minimal Simplicial Complexes with Homology}
\label{sec:minimal-complexes}

Example~\ref{ex:intro-virtual} gives one example of how a virtual resolution with homology can arise from a labeled simpicial complex. In this section, we give lower bounds on the size of labeled simplicial complexes supporting virtual resolutions with homology. These constraints show that the examples arising from labelings of a bipyramid are the smallest labeled simplicial complexes which give rise to virtual resolutions with homology.
As a starting point, on $X=\mathbb{P}^n\times \mathbb{P}^k$, we note that the bipyramid complex $\bp^k$ can always be labeled to give a virtual resolution with homology.

\begin{prop}
  \label{prop:bipyramid-labeling}
  Fix $X=\mathbb{P}^n\times \mathbb{P}^k$ for integers $n$ and $k$ with $n\geq k\geq 0$, with Cox ring $S=\field[x_0,\ldots,x_n,y_0,\ldots,y_k]$. If $\Delta=\bp^k$ is the bipyramid over a $k$-simplex, then there exists a labeling $\ell$ of $\Delta$ by monomials in $S$ that such that $F_\Delta$ is a virtual resolution and not a free resolution.
\end{prop}
\begin{proof}
  Let $v_0,\dots, v_k$ be the vertices of the distinguished $k$-simplex in $\bp^k$ and the other two vertices be $w_0$ and $w_1$.

  Consider the lcm-labeling given by $\ell(v_i)=y_i$, and $\ell(w_0)=x_0$, $\ell(w_1)=x_1$. $F_\Delta$ is not a free resolution since $\tilde{H}_0(\Delta_{x_0x_1};\field)=\field\neq 0$. But any $m\in B$ is divisible by $y_i$ for some $i=1,\ldots,k$. Thus there exists $i=1,\ldots,k$ such that $\Delta_m$ contains $v_i$. As every induced subcomplex that contains any of the $v_1,\dots, v_k$ is contractible, we find that $\redH_j(\Delta_m;\field)=0$ for all $m\in B$ and all $j\geq 0$. Thus by Theorem~\ref{thm:virtual-resolution}, $F_\Delta$ is a virtual resolution.
\end{proof}

The simple examples described in Proposition~\ref{prop:bipyramid-labeling} also represent the smallest examples, for products of two projective spaces, of a virtual resolution with homology. More generally, for a smooth toric variety we can characterize the smallest such examples by the codimension of the irrelevant ideal.

For the remaining proofs in this section it will be useful to consider, for a subcomplex $\Delta_m$, the ideal generated by monomials $f$ for which the complex $\Delta_{fm}$ has no homology.
\begin{definition}
  \label{def:delta-annihilator}
  Given a labeled simplicial complex $\Delta$ and a monomial $m$, define the following monomial ideal in $S$:
  \[I(\Delta,m):=\langle f :  f \text{ monomial and } \redH_i(\Delta_{fm};\field)=0\text { for all } i\geq 0 \rangle.\]
\end{definition}

In the case where $\Delta_m$ is maximal with respect to inclusion among subcomplexes with homology, $I(\Delta,m)$ will in fact contain only the monomials $f$ for which $\redH_i(\Delta_{fm};\field)=0$.
With this definition, we can now prove our first result constraining the structure of virtual resolutions with homology.

\begin{prop}
  \label{prop:min-vertices}
  Fix a smooth toric variety $X$ with Cox ring $S$ and irrelevant ideal $B$.
  Set $c$ to be the codimension of the irrelevant ideal $B$ inside of the affine cone over $X$.
  Let $(\Delta,\ell)$ be an lcm-labeled simplicial complex.
  If $F_\Delta$ is a virtual resolution and there exists $i>0$ such that $H_i(F_\Delta)\neq 0$ then $\Delta$ contains at least $c+\left|\Delta_m\right|$ vertices.
\end{prop}
\begin{proof}
  Suppose that $F_\Delta$ is a virtual resolution and suppose that $H_i(F_\Delta)\neq 0$ for some $i>0$. By Theorem~\ref{thm:cellular-subcomplexes}, there exist some monomial $m$ such that $\redH_i(\Delta_m;\field)\neq 0$. Assume that $m$ is such that the complex $\Delta_m$ is maximal under inclusion among subcomplexes with this property.

  Now let $I=I(\Delta,m)$ as given in Definition~\ref{def:delta-annihilator}.
  Since $F_\Delta$ is a virtual resolution, by Corollary~\ref{cor:annihilator-of-homology} there exists some $d\geq 0$ such that $B^{[d]}\subseteq I$.
  However, as $B$ is codimension $c$, $I$ must also be codimension at least $c$. As a consequence, $I$ must contain at least $c$ minimal generators.

  Now let $m_1,\ldots,m_c\in I$ be a collection of $c$ distinct minimal generators of $I$.
  As these are minimal generators, for each pair $1\leq i,j\leq c$ with $i\neq j$, it follows that $\gcd(m_i,m_j)\notin I$.
  As $\Delta_m$ is maximal among subcomplexes satisfying $\redH_i(\Delta_m;\field)\neq 0$,
  if $\Delta_{fm}\neq \Delta_m$ then $\redH_i(\Delta_{fm};\field)=0$ and thus $f\in I$.
  As $m_1,\ldots,m_c$ are distinct minimal generators of $I$, for each pair of integers $1\leq i,j\leq c$ with $i\neq j$,
  $\gcd(m_i,m_j)\not\in I$. It follows that $\Delta_{\gcd(m_i,m_j)m}=\Delta_{m}$.

  On the other hand, $\Delta_{\gcd(m_i,m_j)m} = \Delta_{m_i\cdot m}\cap \Delta_{m_j\cdot m}$ and thus $\Delta_{m_i\cdot m}\cap \Delta_{m_j\cdot m}=\Delta_m$.
  As $\Delta_{m_i\cdot m}\neq \Delta_m$ it follows that each the sets of vertices of $\Delta_{m_i\cdot m}\setminus \Delta_m$ for $1\leq i\leq c$,
  are non-empty and pairwise disjoint.
  As a consequence, there must be at least $c$ distinct vertices in $\Delta$ that are not in $\Delta_m$.
  Thus $\Delta$ must contain at least $c+|\Delta_m|$ vertices.
\end{proof}

The most important case of Proposition~\ref{prop:min-vertices} is when $\Delta_m$ is a pair of points. Moreover, using a coarse bound on the codimension of $B$, we can get the following uniform bound.

\begin{cor}
  Fix a smooth non-affine toric variety $X$ with $\dim X\geq 1$.
  Let $(\Delta,\ell)$ be an lcm-labeled simplicial complex.
  If $F_\Delta$ is a virtual resolution and there exists $i>0$ with $H_i(F_\Delta)\neq 0$, then $\Delta$ contains at least $4$ vertices.
\end{cor}
\begin{proof}
  As $X$ is not affine, the fan $\Sigma_X$ corresponding to $X$ must contain at least two maximal cones.
  As the generators of the irrelevant ideal are in one-to-one correspondence with the maximal cones in $\Sigma_X$,
  the irrelevant ideal of $X$ is codimension at least $2$, and so the result follows from Proposition~\ref{prop:min-vertices}. 
\end{proof}

Even independent of the reduction results we prove in Section~\ref{sec:homology-reduction}, Proposition~\ref{prop:min-vertices} gives constraints on the structure of virtual resolutions. It suggests that to get virtual resolutions with homology, we need to give ideals with relatively large numbers of generators, at least in the monomial case.

In the smallest possible case, that is where $\Delta_m$ contains exactly $2$ vertices and $\Delta$ contains $c+2$ vertices, $c+2$ is exactly the number of vertices in the $c$-dimensional bipyramid $\bp^{c-1}$. We next show that in fact in this minimal case, $\Delta$ must be exactly a bipyramid. To this end we need the following lemma, which describes a criterion for a complex on $n+2$ vertices to be $\bp^{n-1}$.

\begin{lemma}
  \label{lem:bipyrimid}
  Let $\Delta$ be an acyclic simplicial complex on $n+2$ vertices named $u,w$ and $v_1,\ldots,v_n$. Suppose that for each $k=1,\ldots,n$, the induced subcomplex of $\Delta$ on $u,w,v_1,\ldots,\widehat{v_{k}},\ldots,v_n$ is a bipyrimid over $\{v_1,\ldots,\widehat{v_{k}},\ldots,v_n\}$. Then $\Delta$ is a bipyrimid over $\{v_1,\ldots,v_n\}$.
\end{lemma}
\begin{proof}
  Let $\Gamma_k$ be the induced subcomplex of $\Delta$ on the vertices $u,w,v_1,\ldots,\widehat{v_{k}},\ldots,v_n$. Note in particular that as a bipyramid, $\Gamma_k$ contains exactly two $(n-1)$-simplices, namely $\{u,v_1,\ldots,\widehat{v_{k}},\ldots,v_n\}$ and $\{w,v_1,\ldots,\widehat{v_{k}},\ldots,v_n\}$. Moreover, for all $k$, $\Gamma_k$ does not contain the edge $\{u,w\}$. As such, $\Delta$ also does not contain the edge $\{u,w\}$. 
  
  We start by showing that the $(n-1)$-skeleton of $\Delta$ has the correct simplices. As such, let $\sigma\subset \{u,w,v_1,\ldots,v_n\}$ be a simplex. Then we have three cases of interest. First, if $u$ and $w$ are both in $\sigma$ then $\sigma\not\in \Gamma_k$. However as $\Gamma_k$ is an induced subcomplex, this implies that $\sigma\not\in\Delta$. Second, if exactly one of $u$ and $w$ is in $\sigma$, then $\sigma$ contains all but one of $v_1,\ldots,v_n$, and as such $\sigma\in \Gamma_k\subseteq \Delta$ for some $k=1,\ldots,n$. Finally, we must consider the case where neither $u$ nor $w$ is in $\sigma$.
  
  Fix $\sigma=\{v_1,\ldots,v_n\}$, the unique $(n-1)$-simplex not containing $u$ and $w$. As $\bigcup_{k=1}^n \Gamma_k$ has non-zero $\redH_{n-1}$ but $\Delta$ is acyclic, we find that $\Delta$ must be at least $n$-dimensional. However, the only $n$-simplices not containing the edge $\{u,w\}$ must contain $\sigma$, and so $\Delta$ must contain $\sigma$. All together, this implies that the $(n-1)$-skeleton of $\Delta$ contains exactly all $(n-1)$-simplices not containing the edge $\{u,w\}$.

  Then $(n-1)$-skeleton of $\Delta$ has $\dim \redH_{n-1}(\Delta^{(n-1)};\field)=2$ and $\Delta$ is acyclic, $\Delta$ must contain at least two $n$-simplices. Observe that the only $n$-simplices compatible with the $n-1$-skeleton of $\Delta$ are the simplices $\{u,v_1,\ldots,v_n\}$ and $\{w,v_1,\ldots,v_n\}$. As a consequence, $\Delta$ must contain both of these $n$-simplices. Thus $\Delta$ is exactly the bipyrimid over $\{v_1,\ldots,v_n\}$.
\end{proof}

With that lemma complete, we can now move on to the main result of this section. As a consequence of Proposition~\ref{prop:min-vertices}, the minimum number of vertices a simplicial virtual resolution with homology must have is $c+2$, where $c$ is the codimension of the irrelevant ideal. We now show that bipyramids are the only simplicial complexes for which labelings exist that achieve this minimum.

\begin{prop}
  \label{prop:unique-min-vertices}
  Let $X$ be a normal toric variety with irrelevant ideal $B$ and let $c$ be the codimension of $B$.
  If $(\Delta,\ell)$ is an lcm-labeled simplicial complex on $c+2$ vertices such that $F_\Delta$ is a virtual resolution, and there exists $i>0$ with $H_i(F_{\Delta})\neq 0$, then $\Delta=\bp^{c-1}$.
\end{prop}
\begin{proof}
  Begin by choosing $m$ such that $\redH_i(\Delta_m;\field)\neq 0$ for some $i\geq 0$ and $I(\Delta,m)$ is prime.

  To show that such an $m$ exists, suppose that $I(\Delta,m)$ is not prime.
  Then there exist monomials $f$ and $g$ such that $fg \in I(\Delta,m)$ with $f\notin I(\Delta,m)$ and $g\notin I(\Delta,m)$.
  Since $f\notin I(\Delta,m)$, it follows that $\redH_i(\Delta_{fm};\field)\neq 0$ for some $i\geq 0$.
  Moreover, a quick computation shows that $I(\Delta,m)\subseteq I(\Delta,fm)$.
  Replace $m$ with $fm$ and repeat until $I(\Delta,m)$ is prime.
  As $S$ is Noetherian this process terminates, so there exists an $m$ for which $I(\Delta,m)$ is prime.
  For simplicity of exposition set $I:=I(\Delta,m)$.

  Since $\redH_i(\Delta_m;\field)\neq 0$ for some $i\geq 0$,  we know $\Delta_m$ has at least $2$ vertices.  Because in addition $\Delta$ has $c+2$ vertices,  Proposition~\ref{prop:min-vertices} implies $\Delta_m$ must contain exactly $2$ vertices.
  Call these two distinguished vertices $u$ and $w$.
  
  As the complex $F_\Delta$ is a virtual resolution and $I$ annihilates all higher homology, it follows that $B^d\subseteq I$. Because $I$ is prime, in fact $B\subseteq I$.
  As $B$ is codimension $c$, $I$ is codimension at least $c$.
  Therefore, as $I$ is a prime monomial ideal of codimension at least $c$, it must contain at least $c$ distinct variables.

  Denote $c$ of the variables in $I$ by $x_1,\ldots,x_c$.
  As each of the $x_i$ is a minimal generator of $I$, each of $\Delta_{x_i m}$ is a distinct subcomplex.
  Moreover $\Delta_{m}\subseteq \Delta_{x_i m}$, thus each $\Delta_{x_i m}$ contains exactly 3 vertices.
  As $\Delta_{x_i m}$ is an acylic simplicial complex on three vertices that is not a simplex,
  $\Delta_{x_i m}$ must be a chain of two edges from $u$ to $w$.
  Since $\Delta$ contains $2+c$ vertices, and $\Delta_m$ contains $2$ vertices,
  each vertex of $\Delta$ is either in $\Delta_m$ or in a unique $\Delta_{x_i m}$.
  Now denote by $v_i$ the unique vertex in $\Delta_{x_i m}$ not in $\Delta_m$.
  
  Now for $\mathcal{I}\subseteq \{1,\ldots,c\}$, define the monomial $m_{\mathcal{I}}:=\prod_{i\in \mathcal{I}}x_i$.
  Then for $i\in \mathcal{I}$, we have $\Delta_{x_i m}\subseteq \Delta_{m_{\mathcal{I}} m}$ and consequently, $v_i\in \Delta_{m_{\mathcal{I}}m}$ for all $i\in I$.
  Conversely since $\ell(v_i) \mid x_i m$ and $\ell(v_i)\nmid m$, if $f$ is a monomial in $S$ such that $\ell(v_i) \mid f m$, then $f$ is divisible by $x_i$.
  Thus if $v_i\in \Delta_{m_{\mathcal{I}}\cdot m}$, then $i\in \mathcal{I}$. Thus the vertices of $\Delta_{m_{\mathcal{I}}\cdot m}$ are exaclty the vertices $u$, $w$ along with $v_i$ for $i\in \mathcal{I}$.

  As each $\Delta_{m_{\mathcal{I}}\cdot m}$ contains more than $2$ vertices for $\mathcal{I}$ non-empty, and $\Delta$ contains $2+c$ vertices,
  by Proposition~\ref{prop:min-vertices}, each $\Delta_{m_{\mathcal{I}}\cdot m}$ is acyclic when $\mathcal{I}$ is non-empty.
  Inductively applying Lemma~\ref{lem:bipyrimid} to the subcomplexes
  $\Delta_{m_{\mathcal{I}}\cdot m}$ for $\mathcal{I}\subseteq \{1,\ldots,c\}$ with $\mathcal{I}\neq \emptyset$,
  we find that the complex $\Delta_{m_{\mathcal{I}} m}$ is a bipyramid over the vertices $\{v_i:i\in \mathcal{I}\}$.
  Finally, for $\mathcal{I}=\{1,\ldots,c\}$ we have $\Delta = \Delta_{m_{\mathcal{I}}\cdot m}$, and so $\Delta$ is the bipyramid $\bp^{c-1}$ as desired.
  \end{proof}

\section{Subdivisions and Homology}
\label{sec:homology-reduction}

We now move on to the proofs of the main theorems of the paper, beginning with the proof of Theorem~\ref{thm:subdivision-virtual}, which we recall states that subject to some divisibility criteria, subdivisions preserve virtual resolutions.

\begin{proof}[Proof of Theorem~\ref{thm:subdivision-virtual}]
  Let $(\Delta,\ell)$ be a labeled simplicial complex such that the corresponding complex $F_\Delta$ is a virtual resolution of $S/I$ for a monomial ideal $I$.
  Let $(\Delta',\ell')$ be a virtual compatible subdivision of $(\Delta,\ell)$ at a point $\newvert$ in $\goodsimp$;
  recall that this means the monomial $\ell'(\newvert)$ satisfies two conditions from Theorem~\ref{thm:subdivision-virtual}:
  \begin{enumerate}
  \item
    $\ell'(\newvert)$ divides $\ell(\goodsimp)$, and
  \item
    $\ell'(\newvert)\in \langle \ell(v_0),\ldots,\ell(v_r)\rangle : B^{\infty}$.
  \end{enumerate}

  Condition (\ref{elarge}) immediately implies that $\ell'(\newvert)\in I:B^\infty$.
  Let $J := I+\ell'(\newvert)$, this is the ideal generated by the labels of $(\Delta',\ell')$.
  Note that $I\subseteq J \subseteq I:B^{\infty}$ thus we have $J:B^{\infty}=I:B^{\infty}$.
  Thus if $F_{\Delta'}$ is a virtual resolution, then it is a virtual resolution of $S/I$.
  Thus all that remains to show is that $F_{\Delta'}$ is a virtual resolution.
  By Theorem~\ref{thm:virtual-resolution}, it suffices to find $d\geq 0$ such that $\redH_i(\Delta'_m;\field)=0$ for all monomials $m\in B^d$.

  As the subdivision $(\Delta',\ell')$ is virtual compatible, we can fix $e\geq 0$ such that
  \[\ell'(\newvert)\in \langle \ell(v_0),\ldots,\ell(v_r)\rangle : B^{e}.\]
  Then since $F_\Delta$ is a virtual resolution, Theorem~\ref{thm:virtual-resolution} implies
  that for all $d\gg 0$, we have $\redH_i(\Delta_m;\field)=0$ for all monomials $m\in B^{d}$.
  We will show that if $d$ and $e$ satisfy $d\gg e$ and $m\in B^{d}$, $\redH_i(\Delta'_m;\field)\cong \redH_i(\Delta_m;\field)$ and thus $\redH_i(\Delta'_m;\field)=0$.

  Fix $m\in B^{d}$, if $\ell'(\newvert)$ does not divide $m$, then $\ell(\omega)$ does not divide $m$.
  In this case, $\newvert\notin \Delta'_m$ and $\omega\notin \Delta_m$ so $\Delta'_m=\Delta_m$. But then $\redH_i(\Delta'_m;\field) = \redH_i(\Delta_m;\field)=0$ as desired. 
  We thus restrict to the case where $\ell'(\newvert)$ divides $m$ and thus $\newvert$ is a vertex of $\Delta'_m$. Then define the simplex $\tau$ by
  \[\tau := \{v_i\in \goodsimp : v_i\in \Delta_m\}.\]
  
  If $\tau=\goodsimp$, then $\goodsimp\in \Delta_m$ and $\Delta'_m$ is simply the subdivision of $\Delta_m$ at $\newvert$ in the simplex $\goodsimp$.  A simplicial complex and its subdivision have isomorphic homology, so we have the desired isomorphism $\redH_i(\Delta'_m;\field) \cong \redH_i(\Delta_m;\field)$.

  Otherwise, we may assume that $\tau\neq \goodsimp$.
  The only simplices of $\Delta'_m$ not in $\Delta_m$ are those containing $\newvert$.
  Consequently, $\cpxstar_{\newvert}(\Delta'_m)$ is the subcomplex of $\Delta'_m$ containing all simplices that contain $\newvert$.
  Moreover, as $\tau\neq \goodsimp$ it follows that $\goodsimp \notin \Delta_m$, thus $\Delta_m\subseteq \Delta'_m$.
  Consequently, we get $\Delta'_m = \cpxstar_{\newvert}(\Delta'_m) \cup \Delta_m$, and moreover 
  \begin{align*}
    \cpxstar_{\newvert}(\Delta'_m)\cap \Delta_m &= \{\sigma \in \Delta_m : \sigma \cup \{\newvert\} \in \Delta'_m\}\\
    &= \link_{\newvert}(\Delta'_m).
  \end{align*}

  Observe that the subcomplex $\cpxstar_{\newvert}(\Delta'_m)$ is contractible.
  By Mayer--Vietoris, to show $\redH_i(\Delta'_m;\field)\cong \redH_i(\Delta_m;\field)$ for all $i\geq 0$, it suffices to show that $\cpxstar_{\newvert}(\Delta'_m)\cap \Delta_m$ is non-empty and contractible.  Equivalently, it suffices to show that every maximal simplex of $\link_\newvert(\Delta'_m)$ contains $\tau$ and that $\tau$ is non-empty.

  We begin by showing that $\tau$ is non-empty. For this it suffices to find $0\leq i\leq r$ such that $\ell(v_i)$ divides $m$.
  As $d\gg e\geq 0$, $\ell'(\newvert)\mid m$, and $m\in B^d$, there exists $m'\in B^{e}$ such that $\ell'(\newvert)m'\mid m$.
  As $\ell'(\newvert)\in \langle \ell(v_0),\ldots,\ell(v_r)\rangle : B^{e}$, we find that
  \[\ell'(\newvert)m'\in \langle \ell(v_0),\ldots,\ell(v_r)\rangle.\]
  Thus there exists $0\leq i\leq r$ such that $\ell(v_i)\mid \ell'(\newvert)m'$ and thus $\ell(v_i)\mid m$.
  Consequently, $v_i\in \tau$ and so $\tau$ is non-empty.

  Fix $\sigma\in \link_\newvert(\Delta'_m)$. Then $\sigma\cup \{\newvert\}\in \Delta'_m\subseteq \Delta'$ which implies that $\sigma\cup \goodsimp$ is a simplex of $\Delta$.
  Observe that $(\sigma \cup \tau) \cup \goodsimp = \sigma\cup \goodsimp \in \Delta$ and thus $\sigma\cup \tau\cup \{\newvert\}\in \Delta'$.
  As $\Delta'_m$ is an induced subcomplex of $\Delta'$ and each of $\sigma$, $\tau$, and $\{\newvert\}$ are simplices in $\Delta'_m$
  it follows that $\sigma\cup \tau\cup \{\newvert\}\in \Delta'_m$.
  This implies that $\sigma\cup \tau \in \link_\newvert(\Delta'_m)$, and consequently every maximal simplex of $\link_\newvert(\Delta'_m)$ contains $\tau$.
  Then as $\tau$ is non-empty, we find $\link_\newvert(\Delta'_m)$ is contractible.
  Finally, this implies $\redH_i(\Delta'_m;\field) \cong \redH_i(\Delta_m;\field)$.

  We conclude by observing again that $\redH_i(\Delta_m;\field)=0$ for all $m\in B^{d}$.
  The isomorphism of homology then implies $\redH_i(\Delta'_m;\field) = 0$ for all $m\in B^{d}$ and
  thus $F_{\Delta'}$ is a virtual resolution.
\end{proof}

The content of the above results can be concisely stated in terms of a map of complexes. Let $(\Delta',\ell')$ be a virtual compatible subdivision of $(\Delta,\ell)$ at a point $\newvert$ in $\goodsimp$, then there is a map of complexes $\iota_{\goodsimp}:F_\Delta\rightarrow F_{\Delta'}$ given by
\begin{equation*}
  \iota_{\goodsimp}([\sigma]) := 
  \begin{cases}
    \displaystyle\sum_{\substack{\tau\subset \sigma,~ \sigma\setminus \tau\subset \goodsimp, \\ \left|\tau \right| = \left|\sigma\right| - 1}}\frac{\ell(\sigma)}{\ell'(\tau \cup \{\newvert\})}[\tau\cup \{\newvert\}] & \goodsimp\subseteq \sigma, \\
    [\sigma] & \text{otherwise}.
  \end{cases}
\end{equation*}

It is not immediate that this map is well-defined, as the divisibility relations between $\ell(\sigma)$ and $\ell'(\tau\cup \{\newvert\})$ are not immediate. However, the following lemma states that $\iota_\goodsimp$ is well-defined and injective, and if $F_\Delta$ is a virtual resolution, the induced map on homology is an isomorphism at the level of sheaves.

\begin{lemma}
  \label{lem:induces-sheaf-isomorphism}
Fix a smooth toric variety $X$ with Cox ring $S$. 
  Let $(\Delta,\ell)$ be an lcm-labeled simplicial complex 
  and $(\Delta',\ell')$ be an lcm-labeled virtual compatible subdivision of $(\Delta,\ell)$ at a point $\newvert$ in a simplex $\goodsimp\in \Delta$.
  The map $\iota_{\goodsimp}$ is an injective map of chain complexes, and if $F_\Delta$ is a virtual resolution, then $\iota_{\goodsimp}$ induces an isomorphism $H_i(F_\Delta)^\sim\xrightarrow{\cong} H_i(F_{\Delta'})^\sim$.
\end{lemma}
\begin{proof}
  Fix a simplex $\sigma\in \Delta$ with $\goodsimp \subseteq \sigma$ and let $\newvert$ be the new vertex in $\Delta'$. Then the following cycle is the image of the chain $[\sigma]$ under $\iota_\goodsimp$:
  \[\rho_\sigma := \sum_{\substack{\tau\subset \sigma,~ \sigma\setminus \tau\subset \goodsimp, \\ \left|\tau \right| = \left|\sigma\right| - 1}}\frac{\ell(\sigma)}{\ell'(\tau \cup \{\newvert\})}[\tau\cup \{\newvert\}].\]

  As $\ell'$ is a lcm-labeling of $\Delta'$, the label on $\tau\cup \{\newvert\}$ is given by $\ell'(\tau\cup \{\newvert\}) = \lcm(\ell(\tau),\ell'(\newvert))$.
  Since $(\Delta',\ell')$ is a virtual compatible subdivision $\ell'(\newvert)$ divides $\ell(\goodsimp)$.
  Finally, since $\tau\subseteq \goodsimp$, we have $\ell(\tau)\mid \ell(\goodsimp)$.
  Consequently, the term $\frac{\ell(\goodsimp)}{\ell'(\tau \cup \{\newvert\})}$ is a monomial in $S$. As such the map $\iota_{\goodsimp}$ is a well-defined map $F_\Delta\rightarrow F_{\Delta'}$.

  As the image of the basis of $F_\Delta$ given by the classes associated to the simplices of $\Delta$ is linearly independent, the injectivity of $\iota_{\goodsimp}$ follows.

  Now assume that $F_\Delta$ is a virtual resolution. By Theorem~\ref{thm:subdivision-virtual}, $F_{\Delta'}$ is a virtual resolution and $H_i(F_\Delta)^\sim=H_i(F_{\Delta'})^\sim=0$ for $i>0$. This implies that the induced map $\iota_{\goodsimp})_*$ is trivially an isomorphism for $i>0$.

  For the case of $H_0$, note that $H_0(F_\Delta) = S/I$ where $I$ is the ideal generated by the vertices of $\Delta$. Similarly, we have $H_0(F_{\Delta'})=S/(I+\langle \ell'(\newvert)\rangle)$. Moreover, as the vertex labels of $\Delta$ are all vertex labels of $\Delta'$, the induced map is given by the quotient map $\varphi: S/I\rightarrow S/(I+\langle \ell'(\newvert)\rangle)$. However, as $\ell'(\newvert)\in \langle \ell(v_0),\ldots,\ell(v_r) \rangle : B^\infty \subseteq I:B^{\infty}$, it follows that the kernel of $\varphi$ is $B$-torsion and thus 
  $\varphi$ is an isomorphism on sheaves; see Remark~\ref{rem:vres-saturation}. 

  Thus we have that $\iota_\goodsimp$ induces an isomorphism $H_i(F_\Delta)^\sim\xrightarrow{\cong} H_i(F_{\Delta'})^\sim$ for all $i\geq 0$.
\end{proof}

Theorem~\ref{thm:subdivision-virtual} and the reframing described by Lemma~\ref{lem:induces-sheaf-isomorphism} show that the subdivision procedure provides a technique to construct new virtual resolutions from old. However, we seek constructions that \emph{reduce} this homology. For this purpose more work is needed, since we cannot hope that this local procedure always reduces homology globally. Moreover, a subdivision can in general introduce ``new'' homology, as in the following example.

  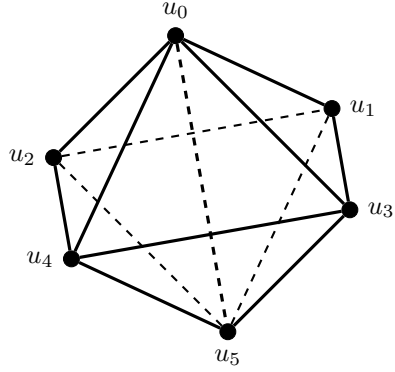
\begin{figure}[ht!]
    \[
    \begin{tikzpicture}
      \pgfmathsetmacro{\angleoffset}{100}
      \pgfmathsetmacro{\radius}{2}

      \node[circle, draw, inner sep=0pt, minimum size=6pt, fill=black, label=above:{$u_0$}] (H1) at (0+\angleoffset: \radius) {};
      \node[circle, draw, inner sep=0pt, minimum size=6pt, fill=black, label=left:{$u_2$}] (H3) at (70+\angleoffset: \radius) {};
      \node[circle, draw, inner sep=0pt, minimum size=6pt, fill=black, label=left:{$u_4$}] (H5) at (110+\angleoffset: \radius) {};
      \node[circle, draw, inner sep=0pt, minimum size=6pt, fill=black, label=right:{$u_3$}] (H4) at (250+\angleoffset: \radius) {};
      \node[circle, draw, inner sep=0pt, minimum size=6pt, fill=black, label=right:{$u_1$}] (H2) at (290+\angleoffset: \radius) {};
      \node[circle, draw, inner sep=0pt, minimum size=6pt, fill=black, label=below:{$u_5$}] (H6) at (180+\angleoffset: \radius) {};

      \draw[very thick] (H1) -- (H3) -- (H5) -- (H6) -- (H4) -- (H2) -- (H1);
      \draw[very thick] (H1) -- (H5) -- (H4) -- (H1);
      \draw[thick, dashed] (H3) -- (H2) -- (H6) -- (H3);
      \draw[very thick, dashed] (H1) -- (H6);
    \end{tikzpicture}
    \]
    \caption{A simplicial complex with vertices labeled by the names to be used throughout the section}
    \label{fig:sec5-diagram-labels}
  \end{figure}

\begin{example}
  \label{ex:new-homology}

  For this and future examples in this section we consider subdivisions of the simplicial complex $\Delta$ on the left of Figure~\ref{fig:homology-increases-example}, depicted with names for the vertices as in Figure~\ref{fig:sec5-diagram-labels}.

  Let $S=\field[x_0,x_1,y_0,y_1]$ be the Cox ring for $X=\mathbb{P}^1\times\mathbb{P}^1$.
  Then consider the following lcm-labeled simplicial complex $(\Delta,\ell)$ which has $H_1(F_{\Delta})_\alpha \cong \field$ for the fine degree given by $x^\alpha=x_0x_1y_0y_1^2$.

\begin{figure}[b!]
\[
  \begin{tikzpicture}
	\pgfmathsetmacro{\angleoffset}{100}
	\pgfmathsetmacro{\radius}{2}

    \coordinate (H5) at (0+\angleoffset: \radius) {};
    \coordinate (H6) at (70+\angleoffset: \radius) {};
    \coordinate (H4) at (110+\angleoffset: \radius) {};
    \coordinate (H2) at (180+\angleoffset: \radius) {};
    \coordinate (H1) at (250+\angleoffset: \radius) {};
    \coordinate (H3) at (290+\angleoffset: \radius) {};

    \node[circle, draw, inner sep=0pt, minimum size=12pt, fill=blue!50!white, draw=none, opacity=50] at (H1) {};
    \node[circle, draw, inner sep=0pt, minimum size=12pt, fill=blue!50!white, draw=none, opacity=50] at (H6) {};

    \node[circle, draw, inner sep=0pt, minimum size=6pt, fill=black, label=below right:$x_0x_1y_0y_1$] at (H1) {};
    \node[circle, draw, inner sep=0pt, minimum size=6pt, fill=black, label=right: $x_0^2y_0^2$] at (H3) {};
    \node[circle, draw, inner sep=0pt, minimum size=6pt, fill=black, label=left: $x_0x_1^2y_0$] at (H5) {};
    \node[circle, draw, inner sep=0pt, minimum size=6pt, fill=black, label=below: $x_0^3x_1y_1$] at (H4) {};
    \node[circle, draw, inner sep=0pt, minimum size=6pt, fill=black, label=right: $x_0^3x_1y_0$] at (H2) {};
    \node[circle, draw, inner sep=0pt, minimum size=6pt, fill=black, label=left:$y_0y_1^2$] at (H6) {};

    \draw[very thick] (H1) -- (H3) -- (H5) -- (H6) -- (H4) -- (H2) -- (H1);
    \draw[very thick] (H1) -- (H5) -- (H4) -- (H1);
    \draw[thick, dashed] (H3) -- (H2) -- (H6) -- (H3);
    \draw[very thick, dashed] (H5) -- (H2);

  \end{tikzpicture}
\qquad
   \begin{tikzpicture}
   
	\pgfmathsetmacro{\angleoffset}{100}
	\pgfmathsetmacro{\radius}{2}

    \coordinate (H5) at (0+\angleoffset: \radius) {};
    \coordinate (H6) at (70+\angleoffset: \radius) {};
    \coordinate (H4) at (110+\angleoffset: \radius) {};
    \coordinate (H2) at (180+\angleoffset: \radius) {};
    \coordinate (H1) at (250+\angleoffset: \radius) {};
    \coordinate (H3) at (290+\angleoffset: \radius) {};
    \coordinate (U1) at ($0.2*(H5)+0.8*(H4)$) {};

    \draw[line width = 5pt, orange!70!white, opacity=50, line cap=round] (H1) -- (U1);
    \draw[line width = 5pt, orange!70!white, opacity=50] (H6) -- (U1);
    \draw[line width = 5pt, orange!70!white, opacity=50] (H6) -- (H3);
    \draw[line width = 5pt, orange!70!white, opacity=50] (H3) -- (H1);

    \node[circle, draw, inner sep=0pt, minimum size=6pt, fill=black, label=below right:$x_0x_1y_0y_1$] at (H1) {};
    \node[circle, draw, inner sep=0pt, minimum size=6pt, fill=black, label=right: $x_0^2y_0^2$] at (H3) {};
    \node[circle, draw, inner sep=0pt, minimum size=6pt, fill=black, label=left: $x_0x_1^2y_0$] at (H5) {};
    \node[circle, draw, inner sep=0pt, minimum size=6pt, fill=black, label=below: $x_0^3x_1y_1$] at (H4) {};
    \node[circle, draw, inner sep=0pt, minimum size=6pt, fill=black, label=right: $x_0^3x_1y_0$] at (H2) {};
    \node[circle, draw, inner sep=0pt, minimum size=6pt, fill=black, label=left:$y_0y_1^2$] at (H6) {};

    \node[circle, draw, inner sep=0pt, minimum size=6pt, fill=white] at (U1) {};

    \draw[very thick] (H1) -- (H3) -- (H5) -- (H6) -- (H4) -- (H2) -- (H1);
    \draw[very thick] (H1) -- (H5) -- (H4) -- (H1);
    \draw[thick, dashed] (H3) -- (H6) -- (H2) -- (H3);
    \draw[very thick, dashed] (H5) -- (H2);
    \draw (H1) -- (U1) -- (H6);
    \draw[dashed] (U1) -- (H2);
    \node[circle, draw, inner sep=0pt, minimum size=6pt, fill=white, label=above right:$\,x_0x_1y_0y_1$] at (U1) {};

  \end{tikzpicture}
\]
    \caption{This figure depicts a labeled simplicial complex $(\Delta,\ell)$ and its subdivision $(\Delta',\ell')$ which introduces new homology,  despite being a virtual compatible subdivision.}
    \label{fig:homology-increases-example}
\end{figure}
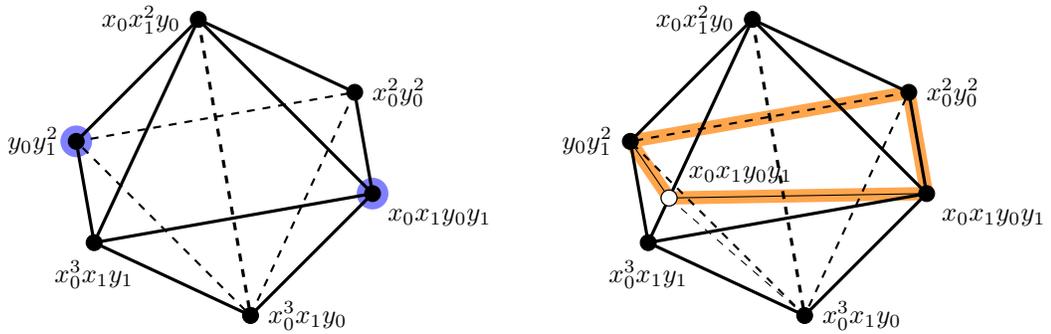

The subdivision of the $1$-simplex $\goodsimp=\{u_0,u_4\}$ may be made virtual compatible by providing the new vertex with the label $x_0x_1y_0y_1$. This satisfies the required conditions:
\begin{enumerate}
\item $x_0x_1y_0y_1$ divides $\ell(\goodsimp)=x_0^3x_1^2y_0y_1$. 
\item If $m\in x_0x_1y_0y_1 \cdot B^3$ then either $x_0^3$ or $x_1^2$ must divide $m$; in the former case $\ell(u_4)\mid m$, and in the latter case $\ell(u_0)\mid m$.
\end{enumerate}
Subdividing at a point in edge between the vertices $u_0$ and $u_4$ yields the lcm-labeled simplicial complex on the right of Figure~\ref{fig:homology-increases-example} which we denote $(\Delta',\ell')$

One checks that the homology previously found in $F_{\Delta}$ is eliminated in $F_{\Delta'}$.  In fact one even sees that $H_1(F_{\Delta'})=0$, for instance by noticing that there are only four vertex sets that induce a non-connected subcomplex, and none of them are of the form $\Delta_m$ for any monomial $m$.

Despite this, $F_{\Delta'}$ is not a free resolution. Indeed, $H_2(F_{\Delta'})_\beta \cong \field$ with degree given by $x^\beta=x_0^2x_1y_0^2y_1^2$, as shown by the shaded cycle above. Interestingly, the original complex has $H_2(F_\Delta)=0$, which can be shown by a similar check of all induced subcomplexes, so our procedure has introduced homology in a higher dimension than existed previously.
\end{example}

As noted earlier, since we are focusing on a small part of a large complex, we cannot expect the subdivision will force the entire complex to become free. Indeed, as just demonstrated in Example~\ref{ex:new-homology}, the resulting complex may even have more homology than the original complex. However, a mild additional divisibility condition forces the subdivsion to eliminate homology at least locally.

\begin{prop}
\label{subdivision-reduces-homology-part1}
Fix a smooth toric variety $X$ with Cox ring $S$.
  Let $(\Delta,\ell)$ be an lcm-labeled simplicial complex such that $F_\Delta$ is a virtual resolution.
  Fix a monomial $m=x^{\alpha}\in S$. Let $\Gamma\subseteq \Delta_m$ be a subcomplex and $\goodsimp=\{v_0,\ldots,v_r\}\in \Delta$ be a simplex such that the join $\Gamma\star\goodsimp$ is a subcomplex of $\Delta$.
  
  Let $(\Delta',\ell')$ be a subdivision of $(\Delta,\ell)$ at a point $\newvert$ in a simplex $\goodsimp\in \Delta$, and suppose that $\ell'(\newvert)\mid m$.
  If $\gamma$ is a class in $H_i(F_{\Delta})_\alpha \cong \redH_{i-1}(\Delta_m;\field)$ supported on $\Gamma$
  then $\gamma$ is in the kernel of the induced map $(\iota_\goodsimp)_*:H_i(F_{\Delta}) \rightarrow H_i(F_{\Delta'})$.
\end{prop}

\begin{proof}
  By assumption $\ell'(\newvert)\mid m$ and so $\newvert\in \Delta'_m$.
  Moreover, $\Gamma\star \newvert\subseteq \Delta'_m$ and $\Gamma\star \newvert$
  is a cone with cone point $\newvert$. Consequently, any homology class supported on $\Gamma$ is trivial in $\Delta'_m$. As $\gamma$ is supported on $\Gamma$ which has no vertices in $\goodsimp$, it follows that $(\iota_\goodsimp)_*(\gamma)=\gamma$, which is thus a trivial class in $\redH_i(\Delta'_m;\field)$. Hence $\gamma$ is in the kernel of the map $(\iota_\goodsimp)_*$ as desired.
\end{proof}

\begin{remark}
  As a consequence of Lemma~\ref{lem:induces-sheaf-isomorphism},  $(\iota_{\goodsimp})_*$ induces an isomorphism of sheaves. However this does not preclude $(\iota_\goodsimp)_*$ from having a non-zero $\gamma$ in its kernel. Such an element $\gamma$ in the kernel need only be $B$-torsion in the appropriate $H_i(F_\Delta)$.
\end{remark}

We illustrate Proposition~\ref{subdivision-reduces-homology-part1} in the following example. This example, unlike the previous examples, has $H_2(F_\Delta)_\alpha \cong \field$ with degree given by $x^\alpha=x_0x_1y_0y_1^2$.

  \begin{figure}[h!]
    \[
    \begin{tikzpicture}
      \pgfmathsetmacro{\angleoffset}{100}
      \pgfmathsetmacro{\radius}{2}

      \coordinate (A) at (0+\angleoffset: \radius) {};
      \coordinate (B) at (70+\angleoffset: \radius) {};
      \coordinate (C) at (110+\angleoffset: \radius) {};
      \coordinate (D) at (250+\angleoffset: \radius) {};
      \coordinate (E) at (290+\angleoffset: \radius) {};
      \coordinate (F) at (180+\angleoffset: \radius) {};

      \draw[line width = 5pt, blue!50!white, opacity=50, line cap=round] (B) -- (C);
      \draw[line width = 5pt, blue!50!white] (C) -- (D);
      \draw[line width = 5pt, blue!50!white] (D) -- (E);
      \draw[line width = 5pt, blue!50!white] (E) -- (B);

      \node[circle, draw, inner sep=0pt, minimum size=6pt, fill=black, label=above:$x_0x_1^2y_0$] (A) at (A) {};
      \node[circle, draw, inner sep=0pt, minimum size=6pt, fill=black, label=left: $x_0x_1y_0y_1$] (B) at (B) {};
      \node[circle, draw, inner sep=0pt, minimum size=6pt, fill=black, label=left: $x_0^2y_0^2$] (C) at (C) {};
      \node[circle, draw, inner sep=0pt, minimum size=6pt, fill=black, label=right: $y_0y_1^2$] (D) at (D) {};
      \node[circle, draw, inner sep=0pt, minimum size=6pt, fill=black, label=right: $x_0^2x_1y_1$] (E) at (E) {};
      \node[circle, draw, inner sep=0pt, minimum size=6pt, fill=black, label=below:$x_0^3x_1$] (F) at (F) {};

      \draw[very thick] (A) -- (B) -- (C) -- (F) -- (D) -- (E) -- (A);
      \draw[very thick, dashed] (A) -- (F);
      \draw[very thick] (A) -- (C) -- (D) -- (A);
      \draw[thick, dashed] (B) -- (F) -- (E) -- (B);

    \end{tikzpicture}
    \qquad\qquad
    \begin{tikzpicture}
      \pgfmathsetmacro{\angleoffset}{100}
      \pgfmathsetmacro{\radius}{2}

      \coordinate (A) at (0+\angleoffset: \radius) {};
      \coordinate (B) at (70+\angleoffset: \radius) {};
      \coordinate (C) at (110+\angleoffset: \radius) {};
      \coordinate (D) at (250+\angleoffset: \radius) {};
      \coordinate (E) at (290+\angleoffset: \radius) {};
      \coordinate (F) at (180+\angleoffset: \radius) {};
      \coordinate (G) at ($0.6*(A)+0.4*(F)$) {};
      
      \draw[line width = 5pt, blue!30!white, opacity=50, line cap=round] (B) -- (C);
      \draw[line width = 5pt, blue!30!white] (C) -- (D);
      \draw[line width = 5pt, blue!30!white,] (D) -- (E);
      \draw[line width = 5pt, blue!30!white,] (E) -- (B);

      \node[circle, draw, inner sep=0pt, minimum size=6pt, fill=white] (G) at (G) {};
      \draw[fill=blue!30!white, fill opacity=0.50] (B) -- (C) -- (D) -- (E) -- cycle;
      \node[label= below right:$x_0^2x_1y_0$] (G) at (G) {};
      
      \node[circle, draw, inner sep=0pt, minimum size=6pt, fill=black, label=above:$x_0x_1^2y_0$] (A) at (A) {};
      \node[circle, draw, inner sep=0pt, minimum size=6pt, fill=black, label=left: $x_0x_1y_0y_1$] (B) at (B) {};
      \node[circle, draw, inner sep=0pt, minimum size=6pt, fill=black, label=left: $x_0^2y_0^2$] (C) at (C) {};
      \node[circle, draw, inner sep=0pt, minimum size=6pt, fill=black, label=right: $x_0^2y_0y_1^2$] (D) at (D) {};
      \node[circle, draw, inner sep=0pt, minimum size=6pt, fill=black, label=right: $x_0^2x_1y_1$] (E) at (E) {};
      \node[circle, draw, inner sep=0pt, minimum size=6pt, fill=black, label=below:$x_0^3x_1$] (F) at (F) {};

      \draw[very thick] (A) -- (B) -- (C) -- (F) -- (D) -- (E) -- (A);
      \draw[very thick, dashed] (A) -- (F);
      \draw[very thick] (A) -- (C) -- (D) -- (A);
      \draw[thick, dashed] (B) -- (G) -- (D);
      \draw[dashed] (C) -- (G) -- (E);
      \draw[dashed] (B) -- (F) -- (E) -- (B);

    \end{tikzpicture}
    \]
    \caption{This figure depicts a labeled simplicial complex $(\Delta,\ell)$ and example of a subdivision $(\Delta', \ell')$ applied to a case where Proposition~\ref{subdivision-reduces-homology-part1} applies.}
    \label{fig:higher-dim-gamma-example}
  \end{figure}
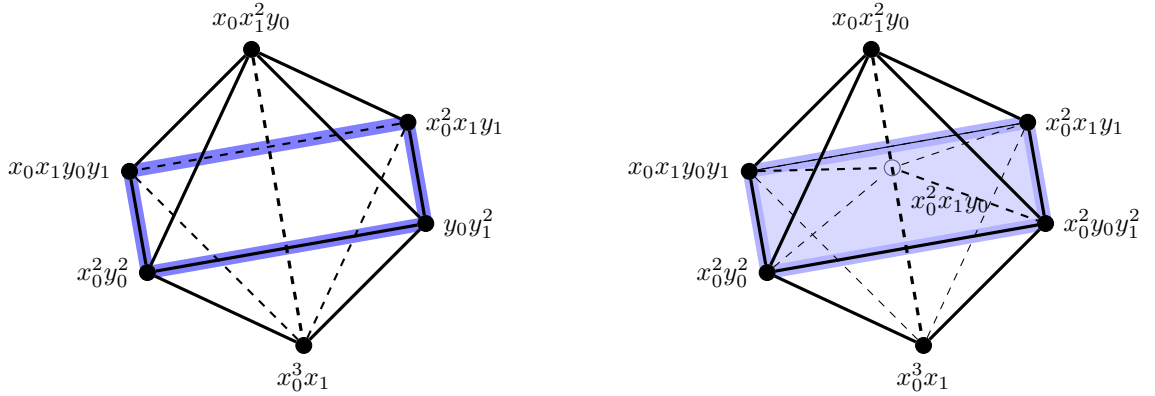

\begin{example}
  Let $S=\field[x_0,x_1,y_0,y_1]$ be the Cox ring for $X=\mathbb{P}^1\times\mathbb{P}^1$.
  In Figure~\ref{fig:higher-dim-gamma-example}, we draw a labeled simplicial complex $(\Delta,\ell)$ and the corresponding virtual compatible subdivision along the simplex $\goodsimp=\{u_0,u_5\}$ with the new vertex $v$ labeled $x_0^2x_1y_0$.

Setting $\Gamma$ to be the subcomplex induced by the vertex set $\{u_1,u_2,u_3,u_4\}$,  we may use Proposition~\ref{subdivision-reduces-homology-part1} with the 1-simplex $\{u_0,u_5\}$.  The label $x_0^2x_1y_0$ of the new vertex satisfies the conditions since
\begin{enumerate}
\item $x_0^2x_1y_0$ divides both $\ell(\goodsimp)=x_0^3x_1^2y_0$ and $\lcm(\ell(u_1),\ell(u_2),\ell(u_3),\ell(u_4))=x_0^2x_1y_0^2y_1^2$.
\item If $m\in x_0^2x_1y_0 \cdot B^1$ then either $x_0^3$ or $x_1^2$ must divide $m$; in the former case $\ell(u_0)\mid m$, and in the latter case $\ell(u_5)\mid m$.
\end{enumerate}

Thus the homology found in $F_\Delta$ is eliminated in $F_{\Delta'}$. In fact, $F_{\Delta'}$ is a free resolution. Any induced subcomplex containing the new vertex $\newvert$ has it as a cone point, and thus by comparison with $\Delta$ one sees that the only subsets of the vertices of $\Delta'$ which induce subcomplexes with homology are those that contain at least one pair of antipodal points from the vertices $u_0,\ldots,u_5$. However, if the subcomplex $\Delta_m$ contains any of the pairs of antipodal points, then $m$ is divisible by $x_0^2x_1y_0$ and thus $\Delta_m$ contains the new vertex $\newvert$. Hence by \cite{bps-monomial-resolutions}*{Lemma~2.2} the complex $F_{\Delta'}$ is a free resolution.
\end{example}

For the proof of Theorem~\ref{thm:main-theorem}, we will still want some control over the global behavior of the homology from this local operation. This we remedy with the following proposition, which describes when the induced map defines a surjection on homology in certain degrees.

\begin{prop}
  \label{subdivision-reduces-homology-part2}
Fix a smooth toric variety $X$ with Cox ring $S$.
  Let $(\Delta,\ell)$ be an lcm-labeled simplicial complex such that $F_\Delta$ is a virtual resolution.
  Let $(\Delta',\ell')$ be an lcm-labeled subdivision of $(\Delta,\ell)$ at a point $\newvert$ in a simplex $\goodsimp\in \Delta$. Write $\Gamma=\link_\goodsimp(\Delta)$.

  Fix a monomial $m=x^{\alpha}\in S$ and an integer $i>0$.
  Then the induced map on homology in degree $\alpha$,
  \[(\iota_{\goodsimp})_*:H_i(F_{\Delta})_{\alpha} \rightarrow H_i(F_{\Delta'})_{\alpha},\]
  is an isomorphism if $\ell(v)$ divides $m$ for some $v\in\goodsimp$ or if $\ell'(\newvert)$ does not divide $m$.
  Otherwise $(\iota_{\goodsimp})_*$ is a surjection if and only if the map of homology $\redH_{i-2}(\Gamma_m;\field)\rightarrow \redH_{i-2}(\Delta_m; \field)$ induced by the inclusion is injective.
\end{prop}

\begin{proof}
  We begin by reframing the induced map $(\iota_{\goodsimp})_*$ in terms of the homology of simplicial complexes.
  By Theorem~\ref{thm:cellular-subcomplexes} we have the following isomorphisms of vector spaces:
  \[H_i(F_{\Delta'})_\alpha \cong \redH_{i-1}(\Delta'_m;\field)\text{ and } H_i(F_{\Delta})_{\alpha} \cong \redH_{i-1}(\Delta_m;\field).\]
  Moreover, there is an identification between the degree $\alpha$ strand of the complex $F_\Delta$ with the reduced simplicial chain complex with coefficients in $\field $ for the simplicial complex $\Delta_m$ and respectively for $F_{\Delta'}$ and $\Delta'_m$.

  Next, we observe that the map  $\iota_{\goodsimp}$ respects the $\ZZ^n$-graded structure on $F_{\Delta}$ and $F_{\Delta'}$ and so is compatible with the above identification. As such, corresponding to a monomial $m$ we get a map of simplicial chain complexes $\iota_{\goodsimp,m}:C_i(\Delta_m;\field)\rightarrow C_i(\Delta'_m;\field)$ given by
  \[
  \iota_{\goodsimp,m}([\sigma]) = 
  \begin{cases}
   \displaystyle\sum_{\substack{\tau\subset \sigma,~ \sigma\setminus \tau\subset \goodsimp, \\ \left|\tau \right| = \left|\sigma\right| - 1}}[\tau\cup \{\newvert\}] & \goodsimp\subseteq \sigma,\\
    [\sigma] & \text{otherwise}.
  \end{cases}
  \]

  Then we wish to show that $\iota_{\goodsimp,m}$ induces either an isomorphism or a surjection on homology, subject to the appropriate conditions.

  If $\ell'(\newvert)$ does not divide $m$ then there is nothing to prove, since $\Delta_m=\Delta'_m$. If $\goodsimp$ is a simplex in $\Delta_m$, then the topology is unchanged by the subdivision, and the map $\iota_{\goodsimp,m}$ induces an isomorphism on the homologies of $\Delta_m$ and $\Delta_m'$ as desired.
  Otherwise, assume $\ell'(\newvert)$ divides $m$ and $\goodsimp\notin\Delta_m$. As in the proof of Theorem~\ref{thm:subdivision-virtual}, it follows that \[\Delta'_m=\Delta_m\cup \cpxstar_{\newvert}(\Delta'_m).\]

  Applying Mayer--Vietoris to this union, we get the following long exact sequence of homology:
\[
\cdots
\rightarrow
\begin{matrix}
  \redH_{i-1}(\Delta_m) \\[-3pt]
  \oplus \\[-3pt]
  \redH_{i-1}(\cpxstar_\newvert(\Delta'_m))
\end{matrix}
\xrightarrow{}
\redH_{i-1}(\Delta'_m)
\xrightarrow{\partial}
\redH_{i-2}(\Delta_m\cap \cpxstar_\newvert(\Delta'_m))
\rightarrow
\begin{matrix}
  \redH_{i-2}(\Delta_m) \\[-3pt]
  \oplus \\[-3pt]
  \redH_{i-2}(\cpxstar_\newvert(\Delta'_m))
\end{matrix}
\rightarrow \cdots.\]

As $\cpxstar_\newvert(\Delta'_m)$ is a cone and thus is acyclic, the complex reduces to
\begin{equation}
  \label{eq:mv-complex}
  \cdots \rightarrow \redH_{i-1}(\Delta_m) \xrightarrow{\iota_{\goodsimp,m}} \redH_{i-1}(\Delta'_m)\xrightarrow{\partial} \redH_{i-2}(\Delta_m\cap \cpxstar_\newvert(\Delta'_m)) \rightarrow \redH_{i-2}(\Delta_m) \rightarrow \cdots.
\end{equation}

Then we simplify $\Delta_m\cap \cpxstar_\newvert(\Delta'_m)$ as follows:
\begin{align*}
\Delta_m\cap \cpxstar_\newvert(\Delta'_m)
  &= \{\sigma\in\Delta_m: \sigma\cup\{\newvert\}\in \Delta_m'\} \\
  &= \{\sigma\in\Delta_m: \lcm(\ell(\sigma),\ell'(\newvert))\mid m ,~ \goodsimp\not\subseteq\sigma, \text{ and } \sigma\cup\goodsimp\in\Delta \} \\
  &= \{\sigma\in\Delta_m: \sigma\cup\goodsimp\in\Delta \}.
\end{align*}
The last equality holds since $\ell(\sigma)$ and $\ell'(\newvert)$ divides $m$ and $\goodsimp\not\in \Delta_m$.

Now suppose further that there exists a vertex $v\in\goodsimp$ such that $v\in \Delta_m$. Then since $\Delta_m$ is an induced subcomplex, there exists a unique maximal simplex $\tau\subseteq \goodsimp$ such that $\tau\in \Delta_m$ and $\tau\neq \emptyset$. Then observe that if $\sigma\cup \omega\in \Delta$ then $(\sigma\cup \tau)\cup \omega \in \Delta$ and so every maximal simplex in $\Delta_m\cap \cpxstar_{\newvert}(\Delta'_m)$ contains $\tau$ and thus $\Delta_m\cap \cpxstar_{\newvert}(\Delta'_m)$ is contractible. By exactness of \eqref{eq:mv-complex} this shows $(\iota_{\goodsimp,m})_*$ is an isomorphism if $\ell(v)$ divides $m$ for some $v\in \goodsimp$.

  Conversely, if no vertex $v\in\goodsimp$ is in $\Delta_m$, then $\Delta_m\cap \cpxstar_\newvert(\Delta'_m)=\Gamma_m$. 
  Then by exactness of \eqref{eq:mv-complex}, $(\iota_{\goodsimp,m})_*$ is a surjection if and only if $\partial=0$. Similarly $\partial=0$ if and only if the inclusion $\redH_{i-2}(\Gamma_m;\field) \rightarrow \redH_{i-2}(\Delta_m;\field)$ is injective. This yields the desired result: if $\ell(v)$ does not divide $m$ for any $v\in\goodsimp$ then $(\iota_{\goodsimp,m})$ is surjective if and only if the map $\redH_{i-2}(\Gamma_m;\field) \rightarrow \redH_{i-2}(\Delta_m;\field)$ is injective.
\end{proof}

Together the results in Proposition~\ref{subdivision-reduces-homology-part1} and Proposition~\ref{subdivision-reduces-homology-part2} form the proof of Theorem~\ref{thm:main-theorem}.

\begin{proof}[Proof of Theorem~\ref{thm:main-theorem}]
  The non-strict inequality follows directly from Proposition~\ref{subdivision-reduces-homology-part2}.

  For the case where $\redH_{i-1}(\Delta_m;\field)$ contains a non-zero class supported on $\Gamma$, we can apply Proposition~\ref{subdivision-reduces-homology-part1} and Proposition~\ref{subdivision-reduces-homology-part2} together to get the strict inequality.
\end{proof}

We conclude this section with the following example that demonstrates Theorem~\ref{thm:main-theorem} applied to a more complicated case.

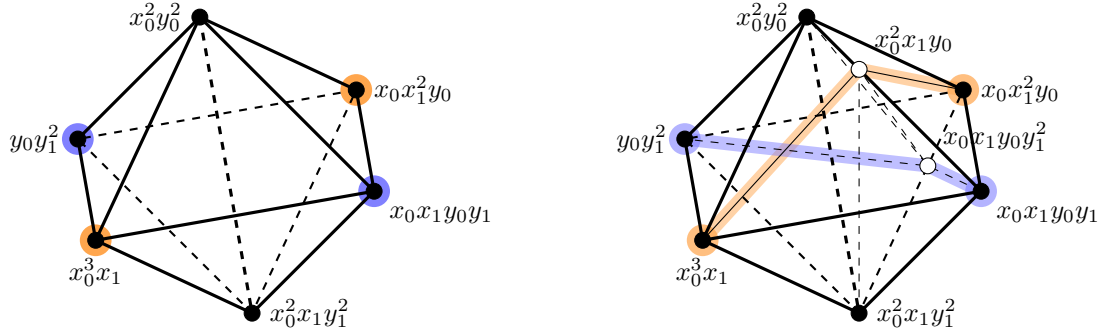
\begin{figure}[h!]
  \[
  \begin{tikzpicture}
    \pgfmathsetmacro{\angleoffset}{100}
    \pgfmathsetmacro{\radius}{2}

    \coordinate (H5) at (0+\angleoffset: \radius) {};
    \coordinate (H6) at (70+\angleoffset: \radius) {};
    \coordinate (H4) at (110+\angleoffset: \radius) {};
    \coordinate (H2) at (180+\angleoffset: \radius) {};
    \coordinate (H1) at (250+\angleoffset: \radius) {};
    \coordinate (H3) at (290+\angleoffset: \radius) {};

    \node[circle, draw, inner sep=0pt, minimum size=12pt, fill=blue!50!white, draw=none, opacity=50] at (H1) {};
    \node[circle, draw, inner sep=0pt, minimum size=12pt, fill=blue!50!white, draw=none, opacity=50] at (H6) {};
    \node[circle, draw, inner sep=0pt, minimum size=12pt, fill=orange!70!white, draw=none, opacity=50] at (H3) {};
    \node[circle, draw, inner sep=0pt, minimum size=12pt, fill=orange!70!white, draw=none, opacity=50] at (H4) {};

    \node[circle, draw, inner sep=0pt, minimum size=6pt, fill=black, label=below right:$x_0x_1y_0y_1$] at (H1) {};
    \node[circle, draw, inner sep=0pt, minimum size=6pt, fill=black, label=right: $x_0x_1^2y_0$] at (H3) {};
    \node[circle, draw, inner sep=0pt, minimum size=6pt, fill=black, label=left: $x_0^2y_0^2$] at (H5) {};
    \node[circle, draw, inner sep=0pt, minimum size=6pt, fill=black, label=below: $x_0^3x_1$] at (H4) {};
    \node[circle, draw, inner sep=0pt, minimum size=6pt, fill=black, label=right: $x_0^2x_1y_1^2$] at (H2) {};
    \node[circle, draw, inner sep=0pt, minimum size=6pt, fill=black, label=left:$y_0y_1^2$] at (H6) {};

    \draw[very thick] (H1) -- (H3) -- (H5) -- (H6) -- (H4) -- (H2) -- (H1);
    \draw[very thick] (H1) -- (H5) -- (H4) -- (H1);
    \draw[thick, dashed] (H3) -- (H6) -- (H2) -- (H3);
    \draw[very thick, dashed] (H5) -- (H2);

  \end{tikzpicture}
  \qquad\qquad
  \begin{tikzpicture}
    \pgfmathsetmacro{\angleoffset}{100}
    \pgfmathsetmacro{\radius}{2}

    \coordinate (H5) at (0+\angleoffset: \radius) {};
    \coordinate (H6) at (70+\angleoffset: \radius) {};
    \coordinate (H4) at (110+\angleoffset: \radius) {};
    \coordinate (H2) at (180+\angleoffset: \radius) {};
    \coordinate (H1) at (250+\angleoffset: \radius) {};
    \coordinate (H3) at (290+\angleoffset: \radius) {};

    \coordinate (U1) at ($0.3*(H1)+0.7*(H5)$) {};
    \coordinate (U2) at ($0.66*(H3)+0.34*(H2)$) {};
    
    \node[circle, draw, inner sep=0pt, minimum size=12pt, fill=blue!30!white, draw=none, opacity=50] at (H1) {};
    \node[circle, draw, inner sep=0pt, minimum size=12pt, fill=blue!30!white, draw=none, opacity=50] at (H6) {};
    \node[circle, draw, inner sep=0pt, minimum size=12pt, fill=orange!50!white, draw=none, opacity=50] at (H3) {};
    \node[circle, draw, inner sep=0pt, minimum size=12pt, fill=orange!50!white, draw=none, opacity=50] at (H4) {};

    \draw[line width = 5pt, orange!70!white,opacity=0.50] (H3) -- (U1);
    \draw[line width = 5pt, orange!70!white,opacity=0.50] (U1) -- (H4);
    \draw[line width = 5pt, blue!50!white,opacity=0.50] (H1) -- (U2);
    \draw[line width = 5pt, blue!50!white,opacity=0.50] (U2) -- (H6);

    \node[circle, draw, inner sep=0pt, minimum size=6pt, fill=black, label=below right:$x_0x_1y_0y_1$] at (H1) {};
    \node[circle, draw, inner sep=0pt, minimum size=6pt, fill=black, label=right: $x_0x_1^2y_0$] at (H3) {};
    \node[circle, draw, inner sep=0pt, minimum size=6pt, fill=black, label=left: $x_0^2y_0^2$] at (H5) {};
    \node[circle, draw, inner sep=0pt, minimum size=6pt, fill=black, label=below: $x_0^3x_1$] at (H4) {};
    \node[circle, draw, inner sep=0pt, minimum size=6pt, fill=black, label=right: $x_0^2x_1y_1^2$] at (H2) {};
    \node[circle, draw, inner sep=0pt, minimum size=6pt, fill=black, label=left:$y_0y_1^2$] at (H6) {};

    \draw[very thick] (H1) -- (H3) -- (H5) -- (H6) -- (H4) -- (H2) -- (H1);
    \draw[very thick] (H1) -- (H5) -- (H4) -- (H1);
    \draw[thick, dashed] (H3) -- (H6) -- (H2) -- (H3);
    \draw[very thick, dashed] (H5) -- (H2);
    \draw (H3) -- (U1) -- (H4);
    \draw[dashed] (H1) -- (U2) -- (H6);
    \draw[dashed] (H2) -- (U1) -- (U2) -- (H5);

    \node[circle, draw, inner sep=0pt, minimum size=6pt, fill=white, label= above right:$x_0^2x_1y_0$] at (U1) {};
    \node[circle, draw, inner sep=0pt, minimum size=6pt, fill=white, label= above right:$ x_0x_1y_0y_1^2$] at (U2) {};
  \end{tikzpicture}
  \]
  \caption{This figure depicts a labeled simplicial complex $(\Delta,\ell)$ and its two-fold virtual compatible subdivision $(\Delta'',\ell'')$ as arising in Example~\ref{ex:multiple-subdivisions}}
  \label{fig:multiple-subdivisions}
\end{figure}

\begin{example}
  \label{ex:multiple-subdivisions}
  Using the labeled simplicial complex $(\Delta,\ell)$ depicted on the left of Figure~\ref{fig:multiple-subdivisions}, the associated chain complex $F_\Delta$ has $H_1(F_\Delta)_\alpha \cong \field$ in two different degrees $\alpha\in\ZZ^4$, corresponding to the monomials $x_0x_1y_0y_1^2$ and $x_0^3x_1^2y_0$.  By subdividing twice we may obtain the figure on the right.

Explicitly, with $\goodsimp = \{u_0,u_3\}$, subdividing at a point in $\goodsimp$ yields $\Delta'$ and  we can label the point by $x_0^2x_1y_0$.
Checking the conditions of Theorem~\ref{thm:subdivision-virtual}, we find that this is a virtual compatible subdivision.

Observe that $\Gamma$ is the subcomplex of $\Delta$ induced by the vertices $\{u_1,u_4,u_5\}$.  We now check the conditions of Theorem~\ref{thm:main-theorem} for $m$ divisible by the new label $x_0^2x_1y_0$. Note the inclusion $\Gamma_m\hookrightarrow \Delta_m$ induces the identity on homology unless $\Delta_m$ contains either $u_0$, $u_1$ or $u_3$. In the first and last case, $m$ divides the label of a vertex in $\goodsimp$. If instead $u_1$ is a vertex of $\Delta_m$,  then $m$ is divisible by $x_0^2x_1y_0y_1^2$ and thus $\Delta_m$ also contains $u_3$, which is also a vertex of $\goodsimp$. Consequently, Theorem~\ref{thm:main-theorem} constrains the homology to not increase.

A similar check shows that Theorem~\ref{thm:main-theorem} applies to the second subdivision as well,
 and so $\dim H_{i}(F_{\Delta''})_\alpha\leq \dim H_{i}(F_\Delta)_{\alpha}$ for all $\alpha$ and all $i>0$.
  
A tedious but straightforward check shows that for the resulting labeled simplicial complex $(\Delta'',\ell'')$, the complex $F_{\Delta''}$ is free.
For instance, if $\redH_i(\Delta''_m;\field)\neq 0$ for some $i\geq 0$
then one easily checks that $m$ must have positive exponents on all four variables, and be divisible by both $y_1^2$ and $x_0^2$. This leaves eight possible vertex sets, all of which can indeed be realized as $\Delta_m$ for some such $m$, but each of the induced subcomplexes are contractible.
\end{example}

\section{Products of Projective Spaces}
\label{sec:products-case}

In this section, we restrict to the case where $X=\mathbb{P}^n\times \mathbb{P}^k$ to describe a class of examples where the techniques of this paper consistently give free resolutions from virtual resolutions containing homology. The class of such virtual resolutions are those arising from the bipyramid simplicial complex $\bp^k$ that also played a role in Proposition~\ref{prop:unique-min-vertices}.

Throughout this section, the vertices of the simplicial complex $\bp^k$, the bipyramid over a $k$-simplex, are named $v_0,\ldots,v_k,w_0,w_1$, where $v_0,\ldots,v_k$ are the vertices of the distinguished $k$-simplex in $\bp^k$ and $w_0$ and $w_1$ are the two other vertices.

\begin{prop}
  \label{prop:join-vres}
  Fix $X=\mathbb{P}^n\times \mathbb{P}^k$ for integers $n$ and $k$ with $n\geq k\geq 0$, 
  with Cox ring $S=\field[x_0,\ldots,x_n,y_0,\ldots,y_k]$.
  Set $\Delta=\bp^k$ and let $\ell$ be an lcm-labeling of $\Delta$.
  
  Then $F_\Delta$ is a virtual resolution if and only if either
  \begin{enumerate}
  \item there exists some $j$ such that $\ell(v_j)\mid \lcm(\ell(w_0),\ell(w_1))$, or
  \item up to a permutation of the variables there exist positive integers $p_0,\ldots,p_\ymax$ and monomials $m_0,\ldots, m_\ymax$ each dividing $\lcm(\ell(w_0),\ell(w_1))$ such that \[\ell(v_j) = y_j^{p_j}m_j.\]
  \end{enumerate}
\end{prop}

\begin{remark}
Note that if $n>k$ then by ``permutations of the variables'' we allow the indices to be shuffled among the $x_i$ and also separately among the $y_j$. If $n=k$, we also may swap the full set of variables $\{x_i\}$ with the full set $\{y_j\}$.
\end{remark}

\begin{proof}
  Throughout, write $f=\lcm(\ell(w_0),\ell(w_1))$.
  A key observation is that the complex $\Gamma=\{\{w_0\},\{w_1\},\emptyset\}$ is the only induced subcomplex of $\Delta$ with non-trivial reduced homology.  Since $\ell$ is an lcm-labeling, $\Delta_m$ is an induced subcomplex for any monomial $m\in S$ and so $\Delta_m$ has non-trivial homology if and only if $\Delta_m=\Gamma$.
  
  We begin with the reverse direction.
  Assume that $(\Delta,\ell)$ is a labeled simplicial complex satisfying either condition (1) or (2) of Proposition~\ref{prop:join-vres}.
  Suppose first that there exists an integer $0\leq i\leq k$ such that $\ell(v_i)\mid \lcm(\ell(w_0),\ell(w_1))$. If $\Delta_m$ contains $w_0$ and $w_1$ then $\ell(w_0)\mid m$ and $\ell(w_1)\mid m$. As $\ell(v_i)\mid \lcm(\ell(w_0),\ell(w_1))$, it follows that $\ell(v_i)\mid m$ and thus $v_i\in \Delta_m$. Thus there is no $m$ where $\Delta_m=\Gamma$, and thus $\Delta_m$ is always empty or contractible for all $m$. Thus by \cite{bps-monomial-resolutions}*{Lemma~2.2} $F_\Delta$ is a free resolution and thus a virtual resolution.

  Otherwise, we can find $p_0,\ldots,p_k$ be positive integers and
  $m_0,\ldots,m_k$ monomials dividing $f$ such that $\ell(v_i)=y_i^{p_i}m_i$.
  Let $m\in S$ be a monomial where $\Delta_m$ is non-empty with non-trivial homology.
  By Corollary~\ref{cor:annihilator-of-homology}, it suffices to show that for some $d\geq 0$, the complex $\Delta_{mx_i^dy_j^d}$ is contractible for all integers $i,j$ with $0\leq i\leq n$, $0\leq j \leq k$.

  Fix $d=\max(p_0,\ldots,p_k)$ and fix $i,j$ with $0\leq i\leq n$, $0\leq j \leq k$.
  As $\Delta_m$ has non-trivial homology, we have that $\Delta_m=\Gamma$ and thus $f\mid m$.
  Then as $\ell(v_j)=y_j^{p_j}m_j$, and $m_j\mid f$, it follows that $\ell(v_j)\mid mx_i^dy_j^d$.
  Thus we have $v_j\in \Delta_{mx_i^dy_j^d}$, and thus $\Delta_{mx_i^dy_j^d}$ has trivial homology as desired.
  This implies that $F_\Delta$ is a virtual resolution.

  For the forward direction, suppose that $F_\Delta$ is a virtual resolution.
  By Corollary~\ref{cor:annihilator-of-homology} there exists $d\geq 0$
  such that $\Delta_{fx_i^dy_i^d}$ has trivial homology for each $i,j$
  with $0\leq i\leq \xmax$ and $0\leq j\leq \ymax$.
  Suppose that $\Delta$ does not satisfy condition (1) and 
  let $m\in S$ be a monomial such that $\Delta_m=\Gamma$.
  For each $0\leq t \leq k$, set $m_t := \gcd(f,\ell(v_t))$ and $m'_t:=\ell(v_t)/m_t$.
  Because $\Delta$ does not satisfy condition (1), we have $\ell(v_t)\nmid f$ and thus $m'_t\neq 1$.
  Now define for each $0\leq t\leq k$ the following set of pairs of integers: \[T_t:=\{(i,j) : m'_t\text{ divides }x_i^{d'}y_j^{d'}\text{ for some }d'\geq 0\}.\]
  
  Observe that $\gcd\{x_i^{d'}y_j^{d'} : (i,j)\in T_t\}$ is divisible by $m'_t$ and thus isn't a unit. So in $T_t$, either all pairs have the same first coordinate, or all pairs have the same second coordinate. 

  For any pair $(i,j)$, the vertices $w_0$ and $w_1$ are in $\Delta_{fx_i^dy_j^d}$ and $\Delta_{fx_i^dy_j^d}\neq\Gamma$.
  As a consequence, for each of the pairs $(i,j)$, the complex $\Delta_{fx_i^dy_j^d}$ must also contain $v_t$ for some $0\leq t \leq k$. 
  This implies $\ell(v_t)$ divides $fx_i^dy_j^d$.
  As $m_t\cdot m'_t = \ell(v_t)$ and $m_t=\gcd(f,\ell(v_t))$ it follows that $\gcd(f,m_t')=1$.
  As $m'_t$ divides $fx_i^dy_j^d$, it then follows that $m_t'\mid x_i^dy_j^d$.
  Thus for integers $i$ and $j$ with $0\leq i \leq n$, $0\leq j\leq k$, each pair $(i,j)$ is in $T_t$ for at least one $t$.

  As $k\leq n$, if any set $T_t$ contains more than $n+1$ pairs, $m_t'$ would be a unit, thus we have $|T_t| \leq n+1$.
  On the other hand, as $0\leq t \leq k$ and there are $(n+1)(k+1)$ distinct pairs $(i,j)$ each of which must belong to at least one $T_t$, so it follows that $|T_t|= n+1$, and moreover that each pair is in exactly one $T_t$.

  Since $k+1\leq n+1$ and $\gcd\{x_i^dy_j^d : (i,j)\in T_t\}$ is not a unit,
  we find that each $T_t$ must consist of a collection of pairs sharing either the first or the second element of the pair.
  However as either each $T_t$ contains only elements having a common second component, or each $T_t$ contains only elements having a common first component,
  without loss of generality we may assume that $T_j$ contains pairs of the form $(i,j)$ for $0\leq i\leq n$.
  That is to say, $m'_j$ divides $x_i^dy_j^d$ for all $i$, and so $m'_j$ is a power of $y_j$.
  Thus there exists positive integers $p_0,\ldots,p_k$ and monomials $m_0,\ldots,m_k$ such that $\ell(v_i)=y_i^{p_i}m_i$ for all $0\leq i \leq k$.
\end{proof}

As an example, take $k=1$. If $F_\Delta$ is a virtual resolution that is not a free resolution then there exists integers $a,b\geq 0$ and monomials $m_0,m_1,f,g\in S$ with $m_0$ and $m_1$ dividing $\lcm(f,g)$ such that $(\Delta,\ell)$ is the following labeled simplicial complex:
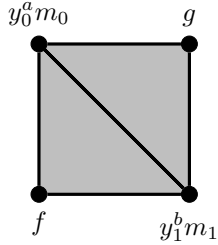
\[
  \begin{tikzpicture}
           \fill[lightgray] (0,0) -- (2,0) -- (0,2) -- cycle;
           \fill[lightgray] (2,2) -- (2,0) -- (0,2) -- cycle;

    \node[circle, draw, inner sep=0pt, minimum size=6pt, fill=black, label=below:$f$] (2) at (0,0) {};
    \node[circle, draw, inner sep=0pt, minimum size=6pt, fill=black, label=above:$g$] (3) at (2,2) {};
    \node[circle, draw, inner sep=0pt, minimum size=6pt, fill=black, label=above:$y_0^am_0$] (4) at (0,2) {};
    \node[circle, draw, inner sep=0pt, minimum size=6pt, fill=black, label=below:$y_1^bm_1$] (5) at (2,0) {};
    \draw[very thick] (5) -- (4) -- (2) -- (5) -- (3) -- (4) -- (5); 
  \end{tikzpicture}
\]

Now consider the subdivision $(\Delta',\ell')$ a point $\newvert$ in the simplex $\{v_0,v_1\}$, where $\newvert$ receives label $m=\lcm(m_0,m_1)$:
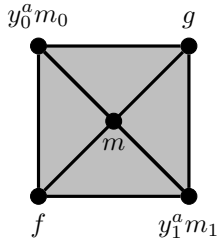
\[
  \begin{tikzpicture}
           \fill[lightgray] (0,0) -- (2,0) -- (0,2) -- cycle;
           \fill[lightgray] (2,2) -- (2,0) -- (0,2) -- cycle;

    \node[circle, draw, inner sep=0pt, minimum size=6pt, fill=black, label=below:$f$] (2) at (0,0) {};
    \node[circle, draw, inner sep=0pt, minimum size=6pt, fill=black, label=above:$g$] (3) at (2,2) {};
    \node[circle, draw, inner sep=0pt, minimum size=6pt, fill=black, label=above:{$y_0^am_0$}] (4) at (0,2) {};
    \node[circle, draw, inner sep=0pt, minimum size=6pt, fill=black, label=below:{$y_1^am_1$}] (5) at (2,0) {};
    \node[circle, draw, inner sep=0pt, minimum size=6pt, fill=black, label=below:{$m$}] (6) at (1,1) {};
        \draw[very thick] (2) -- (3);
        \draw[very thick] (5) -- (4) -- (2) -- (5) -- (3) -- (4) -- (5); 
  \end{tikzpicture}
\]

It is not hard to check that this is in fact a free resolution. This construction can be done in general, using Theorem~\ref{thm:main-theorem} to show that any such labeling of $\bp^k$ that yields a virtual resolution can in fact be subdivided to arrive at a free resolution.

\begin{prop}
  \label{prop:no-higher-homology}
  Fix $X=\mathbb{P}^n\times \mathbb{P}^k$ for integers $n$ and $k$ with $n\geq k\geq 0$, 
  with Cox ring $S=\field[x_0,\ldots,x_n,y_0,\ldots,y_k]$.
  Set $\Delta=\bp^k$ and let $\ell$ be an lcm-labeling of $\Delta$
  such that $F_\Delta$ is a virtual resolution of $S/I$ for a monomial ideal $I\subseteq S$.

  If $(\Delta',\ell')$ is an lcm-labeled virtual compatible subdivision at a point $\newvert$ in $\omega:=\{v_0,\ldots,v_k\}$ satisfying $\ell'(\newvert)\mid \lcm(\ell(w_0),\ell(w_1))$, then $F_{\Delta'}$ is a free resolution of $S/J$, where $J=I + \langle\ell'(\newvert)\rangle$.
\end{prop}

\begin{proof}
  Observe as in the proof of Proposition~\ref{prop:join-vres} that $\Delta_m$ is contractible or empty unless it is $\Gamma=\{\{w_0\},\{w_1\},\emptyset\}$.
  To show that $F_{\Delta'}$ is a free resolution, it suffices to show that for all $\alpha\in \NN^{n+k+2}$, $\dim H_i(F_{\Delta'})_{\alpha}=0$ for $i> 0$.

  Let $m=x^\alpha\in S$ be a monomial, and apply Theorem~\ref{thm:main-theorem} to the subdivsion. Observe that $\Gamma=\link_{\omega}(\Delta)$ as in the statement of Theorem~\ref{thm:main-theorem}.
  Suppose that $v_i\not\in \Delta_m$ for any $0\leq i\leq k$, then $\Delta_m\subseteq \Gamma$. As such, $\Gamma_m=\Delta_m\cap \Gamma = \Delta_m$ and so the inclusion $\Gamma_m\hookrightarrow \Delta_m$ induces an injection on homology in all homological degrees.
  Thus by Theorem~\ref{thm:main-theorem} we get
  \begin{equation}
    \label{eq:ineq-products}
    \dim H_i(F_{\Delta'})_\alpha\leq \dim H_i(F_{\Delta})_\alpha.
  \end{equation}

  Moreover, $H_i(F_{\Delta})_\alpha\neq 0$ if and only if $i=1$ and for the monomial $m=x^{\alpha}$ the subcomplex $\Delta_m$ is exactly $\Gamma$. However, if $\Delta_m=\Gamma$, then $\dim \redH_{0}(\Gamma_m;\field)=1\neq 0$. By Theorem~\ref{thm:main-theorem}, if $\ell'(\newvert)$ divides $m=x^{\alpha}$, then the inequality is strict.

  However, as $\ell'(\newvert)$ divides $\lcm(\ell(w_0),\ell(w_1))$ we find that if $\Delta_m=\Gamma$ then $\ell'(\newvert)$ divides $m$. Thus the inequality \eqref{eq:ineq-products} is strict whenever $H_i(F_\Delta)_\alpha\neq 0$. However, we also know $\dim H_i(F_{\Delta})_\alpha\leq 1$ for all $i>0$ and $\alpha\in \NN^{n+k+2}$. It follows that $\dim H_i(F_{\Delta'})_\alpha = 0$ for all $i>0$ and $\alpha\in \NN^{n+k+2}$.
  
  Consequently $F_{\Delta'}$ has no higher homology and is a free resolution. Finally, $F_{\Delta'}$ resolves $S/J$ for the monomial ideal $J$ generated by the labels on the vertices of $\Delta'$.
  As the only new vertex has label $\ell'(\newvert)$, we have $J=I + \langle\ell'(\newvert)\rangle$. Thus $F_{\Delta'}$ is a free resolution of $S/J$ as desired.
\end{proof}

This applying the results of this section yields a proof of the remaining theorem from the introduction. We recall that Theorem~\ref{thm:star-trivial} roughly seeks virtual resolutions of $S/I$ with vanishing higher homology, when $I$ arises from a virtual resolution on $\bp^k$.
\begin{proof}[Proof of Theorem~\ref{thm:star-trivial}]
  Let $\Delta$ be $\bp^k$, the bipyramid over a simplex.
  By Proposition~\ref{prop:join-vres}, up to a permutation of the variables, there exist integers $p_0,\ldots,p_k$ and monomials $m_0,\ldots,m_k$ not divisible by any $y_j$, such that
  $\ell(v_j)=y_j^{p_j}m_j$.

  Set $m:=\lcm(m_0,\ldots,m_k)$ and let $(\Delta',\ell')$ be the subdivision of $(\Delta,\ell)$ at $\newvert$ and label $\ell'(\newvert)=m$. We can check that this subdivision is a virtual compatible subdivision by checking the conditions in Theorem~\ref{thm:subdivision-virtual}.
  For $\goodsimp=\{v_0,\ldots,v_k\}$ it follows that $\ell(\goodsimp)=\prod_{j=0}^{k}y_j^{p_j}m$. Thus $\ell'(\newvert)\mid \ell(\goodsimp)$.
  Moreover, for $d=\max\{p_0,\ldots,p_k\}$ we get the inclusion
  \[\ell'(\newvert)\cdot \langle y_0,\ldots,y_k\rangle^d\subseteq \langle \ell(v_0),\ldots,\ell(v_k)\rangle.\]
  Since $B\subseteq \langle y_0,\ldots,y_k\rangle$, it follows that
  \[\ell'(\newvert)\cdot B^d\subseteq \langle \ell(v_0),\ldots,\ell(v_k)\rangle.\]
  Thus the monomial $\ell'(\newvert)$ is an element of $\langle \ell(v_0), \ldots ,\ell(v_k) \rangle : B^{\infty}$ as desired.
  
    Then applying Proposition~\ref{prop:no-higher-homology}, we find that the resulting complex $F_{\Delta'}$ is a free resolution. However, as $(\Delta',\ell')$ is a virtual compatible subdivision of $(\Delta,\ell)$ and $F_\Delta$ is a virtual resolution of $S/I$, it follows from  Theorem~\ref{thm:subdivision-virtual} that $F_{\Delta'}$ is a virtual resolution of $S/I$ with $H_i(F_{\Delta'})=0$ for $i>0$. Finally, note that as $\dim \Delta'=\dim \Delta$, the length of the complexes $F_{\Delta'}$ and $F_{\Delta}$ are the same.
\end{proof}

As a consequence of this theorem, even though we can construct labelings on the complex $\bp^k$ with non-trivial homology, there is always a free resolution of the same length that is a virtual resolution of the same module. This work provides new constraints on what modules may have short virtual resolutions with higher homology, though it remains to be seen whether there exist modules whose minimum-length virtual resolutions all have homology.

\section*{Acknowledgments}

The authors would like to thank Christine Berkesch for some early thoughts on the project, as well as the organizers of the 2023 Minnesota Research Workshop in Algebra and Combinatorics, Elise Catania, Sasha Pevzner, and Sylvester Zhang. The authors would also like to acknowledge the other participants in the group at MRWAC including Tomas Banuelos,
Pouya Layeghi,
Joseph McDonough,
Miranda Moore,
Mykola Sapronov,
Eduardo Torres D\'{a}vila,
 and
Mahrud Sayrafi.

\end{document}